\documentclass[10.5pt]{article}
\usepackage{mathtools}
\usepackage[font={small,it}]{caption}

\usepackage{hyperref}
\usepackage{amsmath,amssymb,amstext,amsfonts}
\usepackage{enumerate}
\usepackage{algorithm}
\usepackage{algorithmic}
\usepackage{relsize}
\usepackage{amsthm}
\newcounter{mycounter}
\usepackage[all,arc]{xy}
\usepackage{mathrsfs}
\usepackage{tcolorbox}
\usepackage{graphicx}
\usepackage{subcaption}
\usepackage{float}
\usepackage[margin=0.9 in]{geometry}
\usepackage{listings}
\usepackage{xcolor}
\usepackage{subcaption}

\usepackage{graphicx}
\usepackage[utf8]{inputenc}
\newtheorem{theorem}{Theorem}[section]
\newtheorem{lemma}[theorem]{Lemma}
\newtheorem{proposition}[theorem]{Proposition}

\newtheorem{assumption}{Assumption}
\theoremstyle{definition}

\usepackage{systeme}
\usepackage{comment}
\usepackage{dsfont}
\newtheorem*{remark}{Remark}

\tcbuselibrary{theorems}
\theoremstyle{definition}

\usepackage{sectsty}

\newcounter{texercise}
\newwrite\solout
\def\openoutsol{\immediate\openout\solout\jobname.sol}  \def\writesol#1{\immediate\write\solout{\noexpand\processsol{\thetexercise}{#1}}}

\newcounter{mytheorem}[section] 
\tcbset{
theostyle/.style={fonttitle=\bfseries\upshape, fontupper=\slshape,
arc=0mm, colback=blue!5,colframe=blue!75!black}, defstyle/.style={fonttitle=\bfseries\upshape, fontupper=\slshape,
colback=red!10,colframe=red!75!black},
}

\newcommand{\xu}{X^{u}}

\newcommand{\xuv}{\Delta X^{u,v}}

\newcommand{\zu}{Z^u}
\newcommand{\zv}{Z^v}
\newcommand{\zuv}{\Delta Z^{u,v}}
\newcommand{\yuv}{\Delta Y^{u,v}}
\newcommand{\iuv}{I^{u,v}}

\newcommand{\bR}{\mathbb{R}}

\newcommand{\E}{\mathbb{E}}

\newcommand{\dxu}{DX^u_{t}}

\numberwithin{equation}{section}
\begin{document}
\title{Batch Sample-wise Stochastic Optimal Control via Stochastic Maximum Principle}
\date{}
\author{
Hui Sun  \thanks{Corresponding author. Department of Financial and Actuarial Mathematics, 
School of Mathematics \& Physics. Xi’an Jiaotong-Liverpool University, Suzhou 215123, P.R.China
Email: Hui.Sun@xjtlu.edu.cn.}\ , Feng Bao \thanks{Department of Mathematics, Florida State University, FL 32306, USA. Email:fbao@fsu.edu}}
\date{This Version: \today}
\maketitle

\begin{abstract}
	In this work, we study the stochastic optimal control problem (SOC) mainly from the probabilistic view point, i.e. via the Stochastic Maximum principle (SMP) \cite{Peng4}. We adopt the sample-wise backpropagation scheme proposed in \cite{Hui1} to solve the SOC problem under the strong convexity assumption. Importantly, in the Stochastic Gradient Descent (SGD) procedure, we use batch samples with higher order scheme in the forward SDE to improve the convergence rate in \cite{Hui1} from $\sim \mathcal{O}(\sqrt{\frac{N}{K} + \frac{1}{N}})$ to $\sim \mathcal{O}(\sqrt{\frac{1}{K} + \frac{1}{N^2}})$ and note that the main source of uncertainty originates from the scheme for the simulation of $Z$ term in the BSDE. In the meantime, we note the SGD procedure uses only the necessary condition of the SMP, while the batch simulation of the approximating solution of BSDEs allows one to obtain a more accurate estimate of the control $u$ that minimizes the Hamiltonian. We then propose a damped contraction algorithm to solve the SOC problem whose proof of convergence for a special case is attained under some appropriate assumption. We then show numerical results to check the first order convergence rate of the projection algorithm and analyze the convergence behavior of the damped contraction algorithm. We note that comparisons between the numerical results from these two algorithms (projection and contraction) show the projection approach is still favored in terms of stability and efficiency. Lastly, we briefly discuss how to incorporate the proposed scheme in solving practical problems especially when the Randomized Neural Networks are used. We note that in this special case, the error backward propagation can be avoided and parameter update can be achieved via purely algebraic computation (vector algebra) which will potentially improve the efficiency of the whole training procedure. Such idea will require further exploration and we will leave it as our future work. 
	
\vspace{1em}

\vspace{1em}
\end{abstract}

\section{Introduction}
In recent years, stochastic optimal control problems have been studied extensively especially with the dawning of machine learning/deep learning assisted mathematical approaches. A large body of both theoretical \cite{Bensoussan1_T, Yong1, Peng1, Haussmann, pham1} and applied numerical aspect of the subject \cite{Zhao1, Du1, Peng3, Hanson, Hui2, Kushner} can be found. Recently, an extensive amount of work has been dedicated in leveraging the deep learning techniques to solve the control problems in a supervised learning framework for both the classical and the mean-field settings \cite{Jiequn1, pham2, pham3, Carmona2, Carmona3, Carmona4, Ruimeng1}. All those frameworks rely on directly approximating the control with a deep neural network whose trainable weights are then updated iteratively by a general gradient descent type method via minimizing the expected loss. The current work instead focuses on using one of the more classical probabilistic approach in solving the stochastic optimal control problem, namely the stochastic maximum principle. In this approach, one will first find the numerical approximation of the variational derivative $J'_u(u)$ of the loss and then update the control via gradient descent method. The benefit of such approach is that when one works with a deterministic control under some strong convexity like condition, and when $J'_u(u)$ is appropriately approximated, the first order error convergence can be shown \cite{Zhao1}. The downside however, is that a good approximation of $J'_u(u)$ requires solving a system of FBSDE at each gradient descent iteration which is time consuming. More specifically, finding the numerical approximation $J^{',N}_u(u)$ typically requires solving the following FBSDEs (discretized): 
\begin{align}\label{system_og_num}
\begin{cases}
X^{N}_n&=b(t_n, X^N_n,u_{t_{n}}) h + \sigma(t_n, X^N_n,u_{t_{n}}) \Delta W_n\\ 
Y^N_{t_n}&=\E_{t_n}[Y^N_{t_{n+1}}]+h\E_{t_n}[b'_x(t_n, X_n,u_{t_{n}})Y^N_{{t_{n+1}}}+\sigma'_x(t_n, X_n,u_{t_{n}})Z^N_{t_n}+f'_x(t_n, X_n,u_{t_{n}})]  \\
Z^N_{t_n}&=\E_{t_n}[Y^N_{t_{n+1}}\Delta W_{t_{n}}]/h.
\end{cases}
\end{align}
Common methods used in solving the system of such FBSDE involves i) Monte Carlo methods ii) using regressors to approximate $(Y^N_{t_n},Z^N_{t_n})$. iii) Find $(Y^N_{t_n},Z^N_{t_n})$ on preselected mesh-points and then use interpolation to find approximation for the entire target function. iv) use deep learning to find approximation for both $(Y^N_n,Z^N_n)$ and solve them backward iteratively. However all of these approaches requires finding the solution of the BSDEs iteratively for each projection iteration. And some of the methods are not very effective when the underlying dimension of the control of state is high. 

On the other hand, it is noted in \cite{feng1} and \cite{Hui1} that the goal of finding the optimal control is not about solving the BSDEs but to find an effective approximation of the projection $J'_u(u)$. As such, a sample-wise approximation of the \eqref{system_og_num} is proposed:
\begin{align}
\begin{cases}
X_{n+1} =X_n + b(t_n,X_n, u_n) h + \sigma(t_n, u_n)\Delta W_n  \\ 
Y_n =Y_{n+1}+h \big( b'_x(t_n,X_n, u_n) Y_{n+1} + \sigma'_x(t_n,X_n, u_n) Z_{n}+ f'_x(t_n,X_n, u_n) \big), \ \ Y_T = g_x(X_T)   \\
Z_n = \frac{Y_{n+1} \Delta W_n}{h} 
\end{cases}
\end{align}
and the original gradient (under numerical scheme) $J^{',N}_u(u)|_{t=t_n}= \E[b'_u(t_n,X_t, u_t) Y^{N,u}_{t_n} + \sigma'_u(t_n,X^{N,u}_{t_n}, u_{t_n}) Z^{N,u}_{t_n} + f'_u(t_n,X^{N,u}_{t_n}, u_{t_n})]$ is approximated by the sample gradient $j'_u(u) = b'_u(t,X_n, u_{n}) Y^u_n + \sigma'_u(t,X_n, u_n) Z^u_n + f'_u(t,X_n, u_n)$. This sample-wise approach together with stochastic gradient descent as the update procedure is shown to be very effective. A convergence proof under the strong convexity assumption is also given in \cite{Hui1} and the rate of convergence is found to be $\sim \mathcal{O}(\sqrt{\frac{N}{K} + \frac{1}{N}})$.

We still adopt the sample-wise backward simulation in this work and we seek to improve the efficiency of the algorithm by using batch sampling and provide rigorous proofs for the rate of convergence. We will also explore a new optimization approach based on batch sampling and study its effectiveness.  The main results of the paper are as follows: 
\begin{enumerate}
    \item we adopt batch samples with higher order scheme in the forward SDE to improve the convergence rate in \cite{Hui1} from $\sim \mathcal{O}(\sqrt{\frac{N}{K} + \frac{1}{N}} )$ to $\sim \mathcal{O}(\sqrt{\frac{1}{K} + \frac{1}{N^2}})$ and note that the main source of uncertainty originates from the scheme for the simulation of $Z$ term in the BSDE. Thus, the batch simulation in the BSDE and the SGD procedure significantly reduces the noise in the original estimator $j'_u(u)$. We also note that the first order convergence of the control is the optimal rate given the Euler type approximation of the true control. 
    \item In the meantime, we note the SGD procedure uses only the necessary condition of the SMP, while the batch simulation of the approximating solution of BSDEs allows one to obtain a more accurate estimate of the control $u$ that minimizes the Hamiltonian. We then propose a damped contraction algorithm to solve the SOC problem whose proof of convergence is attained under some appropriate assumption. We then show numerical results to check the first order convergence rate of the projection algorithm and convergence behavior of the damped contraction algorithm. We comment that the convergence analysis for the contraction approach, though needs to be further refined, still sheds some light about how to chose the batch size $M$ against $N$
    
    We note that comparisons between the numerical results from these two new algorithms (batch SGD projection and contraction) and the original SGD shows that the two new approaches based on batch sampling is much more time-efficient and the batch SGD approach stands out to be most efficient and whose experimental results are more interpretable.
    
    \item Lastly, we discussed and showed how to incorporate the proposed scheme in training neural network especially when the Randomized Neural Networks are used. We note that in this special case, the backward propagation can be avoided and parameter update can be achieved via purely algebraic computation (vector algebra) which will potentially improve the efficiency of the whole training procedure. Such idea will require further exploration and we will leave it as our future work. In this regard, this paper also lays out some theoretical foundation for the design of neural network structure and the understanding of training procedure. 
\end{enumerate}

The rest of the paper is structured as follows: the general stochastic optimal control problems and the stochastic maximum principle framework are discussed in Section 2; In Section 3, we narrow the scope down by considering only the time deterministic control problems and state the two main algorithms in \ref{algorithm_batch_sample}. In Section 4, we state additional assumptions and provide proofs for the algorithms. In particular, we show that under the current method, the rate of convergence for the projection algorithm is improved to $\sim \mathcal{O}(\sqrt{\frac{1}{K}+\frac{1}{N^2}})$ which shows the benefit of using the batch samples for gradient update. In Section 5, we provide numerical examples to show convergence behavior of the algorithm. We finally conclude the paper with plan for future work. Sample code can be found on \url{https://github.com/Huisun317/Batch_SMP}.

\section{Review of SOC Problem and the SMP.}
Consider the following controlled stochastic process
\begin{align}
	d X_t=b(t, X_t, u_t) dt + \sigma(t,x_t, u_t) dW_t, \ \  X_{t_0}=x_0
\end{align}
where we have $X_t \in \bR^d, W_t \in \bR^m, b(\cdot,\cdot,\cdot) \in \bR^d,  \sigma(\cdot,\cdot) \in \bR^{d \times m}$, and $\mathcal{U}$ is the collection of all controls $u_t \in \bR^k$ such that 
\begin{align}
	\E[\int^T_0 |b(t, 0, u_t)|^2 +|\sigma(t, 0, u_t)|^2 dt]  < \infty.
\end{align}

To ease notations, we sometimes suppress the time dependencies in the coefficients $b,\sigma, f$. We make the following assumptions: 
\begin{assumption}\label{assumption 1}
  Let $b, \sigma, f, g$ be deterministic smooth functions in their variables with the derivatives $\partial_x b \in \bR^{d\times d},\partial_{u} b \in \bR^{d\times k}, \partial_x \sigma \in \bR^{d\times m \times d}, \partial_{u} \sigma \in \bR^{d \times m \times k}$ all bounded such that :
   \begin{enumerate}[i.]
       \item The lipschitz condition hold uniformly in $t$:  $$|\phi(x,u)-\phi(y,v)| \leq C ( |x-y| + |u-v|).$$
     where $\phi:=b,b'_x,b'_u, \sigma, \sigma'_x, \sigma'_u$.  We also assume that $|b(0,u)|+|\sigma(0,u)| \leq C$ for some $C>0$.

       \item The terminal function $g$ is bounded from below and it has at most quadratic growth. 
       \begin{align}
       	|g(x)| \leq C(1+|x|^2).
       \end{align}
   \end{enumerate}
\end{assumption}
Typically, the control $\mathcal{U}(t_0, x_0)$ contains all the controls evaluated in $U \subset \bR^k$ such that 
       \begin{align}
       	\E[\int^T_{t_0} f(t, X_{t}, u_t) dt] < \infty
       \end{align}
   and we assume that the set $\mathcal{U}(t_0,x_0)$ is not empty for all $(t_0, x_0) \in [0, T] \times \bR^d$. 

Abusing notation by still writing the space of admissible controls as $\mathcal{U}$, we state the stochastic optimal control problem as follows: 
\begin{align}
\inf_{u \in \mathcal{U}} J(u):=\E[\int^T_{t_0} f(t,X_t, u_t) dt + g(X_T)]
\end{align}

\subsection{Stochastic Maximum Principle}
The stochastic optimal control can be classically solved by first using the dynamic programming principle and then the related value function can be attained by solving the derived HJB equation. 
In this section however, we take the probabilistic approach and solve the stochastic optimal control problem by using the Pontryagin's Stochastic Maximum principle. Let $b \in \bR^d, z \in \bR^{d \times m}$ and define the \textit{Hamiltonian} to be: 
\begin{align}
	H(t,x,y,z,u):= b y +\sigma z + f 
\end{align}
where we make the following comments on the Hamiltonian 
\begin{enumerate}[i).]
	\item $y, z$ are the co-variables
	\item The multiplication is understood to be the following
	\begin{enumerate}
	\item $b y := b^T y$
	\item $\sigma z := tr(\sigma^T z)=tr(z^T \sigma) $
	\end{enumerate}
\end{enumerate}

Assuming convex control domain and consider $u^{\epsilon}:= u+\epsilon \beta, \beta:= v-u$ and we consider the following SDE
\begin{align}
	d DX^u_{t}= \bigg(b'_x(t,X^u_t, u_t) \dxu  + b'_u(t,X^u_t, u_t) \beta_t \bigg) dt + \bigg(\sigma'_{x}(t,X^u_t, u_t) \dxu  + \sigma'_u(t,X^u_t, u_t) \beta_t \bigg) dW_t \ DX^u_{0}=0  \nonumber
\end{align}
we note that this SDE is motivated by considering the following derivative: 
\begin{align}
	\lim_{\epsilon \rightarrow 0 }\frac{X_t^{u+\epsilon \beta} - X_t^{u}}{\epsilon}. 
\end{align}
See Lemma 4.7 \cite{carmona1} for a proof of the well-posedness of the solution of the above SDE.

The covariables introduced are defined to be the adjoint process $(Y_t, Z_t)$ which is a solution to the following BSDE: 
\begin{align}
	dY_t = - \partial_x H(t, X_t, Y_t, Z_t, u_t ) dt +Z_t dW_t, \ \ Y_T= \partial_x g(X_T)
\end{align}
where by definition of the Hamiltonian
\begin{align}
	\partial_x H(t, X_t, Y_t, Z_t, u_t)= \partial_x b^T Y_t + tr\big(\partial_x \sigma^T Z_t) + \partial_x f. 
\end{align}
Then, by applying the It\^o product $Y_t^T DX^u_{t}$ and take expectation, the following equation is easily obtained: 
\begin{align}
	\E[\partial_u g(X^u_T)]&=\E[Y^T_T DX^u_{T} ] \nonumber\\ 
	& = \E[ \int^T_0 Y^T_t  b'_u(t,X^u_t, u_t) \beta_t - (\dxu)^T  f'_x(t,X^u_t, u_t) + tr(Z^T_t  \sigma'_u(t,X^u_t, u_t) \beta_t) dt] \label{terminal_derivative}
\end{align}
Now we find gradient of $J$ in the direction of $\beta$: 
\begin{align}
	\lim_{\epsilon \rightarrow 0}\frac{J(u+\epsilon \beta)-J(u)}{\epsilon} = \int^T_0 \E \big[ f'_x(t,X^u_t, u_t) \dxu + f'_u(t,X^u_t, u_t) \beta_t dt  +  g'_x(X^u_T) DX^u_T \big]
\end{align}
By substituting \eqref{terminal_derivative} into the above equation, we obtain: 
\begin{align}
	\frac{d J(u+\epsilon \beta)}{d \epsilon}\bigg|_{\epsilon=0} = \E[\int^T_0 \partial_{u}H(t, X_t, Y_t, Z_t, u_t) \cdot \beta_t  dt ]
\end{align}
As such, suppose that $u^*$ is the optimal control and since the control domain is convex, we have $\forall s \in U$:
\begin{align}
	\partial_{u}H(t, X_t, Y_t, Z_t, u^*_t) \cdot (s- u^*_t) \geq 0 , a.e. \textit{ in } t, \ a.s.
\end{align}
And we note that given the control $u\in \mathcal{U}_{ad}$, 
\begin{align}
    J'_u(u)|_{t=t_n}&= \E[\partial_{u}H(t_n, X_{t_n}, Y_{t_n}, Z_{t_n}, u_{t_n})]      \nonumber\\ 
    &= \E[b'_u(t_n,X_{t_n}, u_{t_n}) Y_{t_n} + \sigma'_u(t_n,X_{t_n}, u_{t_n}) Z_{t_n} + f'_u(t_n,X_{t_n}, u_{t_n})]
\end{align}

Furthermore, we note that with some additional assumption on the convexity of the terminal loss function and the Hamiltonian, one can obtain global minimum.

\begin{theorem}\label{thm_smp}
    With all the above assumptions, assume in addtion that $g$ is convex and that for each $t \in [0,T]$, the function $(x,u) \rightarrow H(t,x,Y_t,Z_t,u)$ is convex. Then if $u^* \in \mathcal{U}_{ad}$  
    \begin{align}
        H(t,X^{u^*}_t,Y^{u^*}_t,Z^{u^*}_t, u_t^*) = \inf_{u \in U} H(t,X^{u^*}_t,Y^{u^*}_t,Z^{u^*}_t, u), a.s
    \end{align}
    then the control $u$ is an optimal control. 
\end{theorem}
For a proof, please refer to \cite{carmona1}, theorem 4.14. 
As such, based on the theorem above, we can write down the following extended Hamiltonian system: 
\begin{align}
    \begin{cases}\label{Ham_system}
        dX^{u^*} = b(t,X_t^{u^*},u_t^*)dt + \sigma(t,X_t^{u^*},u_t^*) dW_t, \ \ &X^{u^*}_0=x_0 \\ 
        dY_t^{u^*}  = - H_x(t, X^{u^*}_t,Y^{u^*}_t,Z^{u^*}_t, {u_t^*}) dt + Z^{u^*}_t dW_t , \ \ &Z^{u^*}_T = \partial_xg(X_T^{u^*}) \\ 
         H(t,X^{u^*}_t,Y^{u^*}_t,Z^{u^*}_t, u_t^*) = \inf_{u \in U} H(t,X^{u^*}_t,Y^{u^*}_t,Z^{u^*}_t, u)
    \end{cases}
\end{align}
In the following, we will suppress the dependency on the control $u$ for the process $(X^{u}_t,Y^{u}_t,Z^{u}_t)$ when its meaning is clear from the context.

\section{The two Algorithm}
In this work, we restrict our control to be the space of \textit{deterministic} controls: 
\begin{align}
 \mathcal{U}_{ad}:= \Big \lbrace u: [0,T] \rightarrow U \big | u \in L_2([0,T]; \bR^k) \Big \rbrace
\end{align}
for some convex set $U \subset \bR^k$. For simplicity, we will work with $d=m=k=1$, higher dimensional generalization is similar.We will work with the whole space $U = \bR^k$, and so we assume $\lim_{||u|| \rightarrow \infty} J(u) \rightarrow \infty$. Also, we let the diffusion term be controlled but it is not a function of the state $X_t$. 

Then according to the stochastic maximum principle, we are considering $u \in  \mathcal{U}_{ad}$ for the following  
\begin{align} \label{ham2}
    \E[H(t,X^{u^*}_t,Y^{u^*}_t,Z^{u^*}_t, u_t^*)] = \inf_{u \in U} \E[H(t,X^{u^*}_t,Y^{u^*}_t,Z^{u^*}_t, u)]. 
\end{align}
\begin{enumerate}
    \item In light of the necessary condition of the SMP, one can consider the projection algorithm since the gradient alone usually provides rich information on the loss functional. Thus, one can propose the following algorithms (see \cite{Hui1, Hui2, Zhao1, feng1}), for each $k \in \lbrace 1, ..., K \rbrace$
\begin{align}
    u_t^{k+1}= u_t^k - \eta_k \nabla_{u} \E[H(t, X^{u^k}_t, Y^{u^k}_t, Z^{u^k}_t, u^k_t)], \ \ \forall t \in [0,T].
\end{align}
where we use the $k$ to denote the $k$th step in the iteration and ${u^k}$ the control related to the FBSDE system. 

\item On the other hand,  to find the optimal control, one can treat \eqref{ham2} as a fixed point problem and tentatively propose the following solution strategy: for $k=1,.., K$, $K \in \mathbb{N}^+$,
\begin{align}
    u_t^{k+1} = \max_{u \in U}\E[H(t,X^{u^k}_t,Y^{u^k}_t,Z^{u^k}_t, u)]
\end{align}
and expect the solution to converge. However, numerical experiments shows that such simple algorithm in most cases will not take place because the map does not form a contraction. Thus, one can use  damping $\rho \in (0,1)$ to ensure that the update is not too rapid: for $k=1,.., K$, $K \in \mathbb{N}^+$,
\begin{align}
    u_t^{k+1} = (1-\rho) \max_{u \in U}\E[H(t,X^{u^k}_t,Y^{u^k}_t,Z^{u^k}_t, u)] + \rho u_t^k. 
\end{align}
We note that the existence of such contraction usually imposes more restrictive structure on the terminal cost $g$ and the structure of the Hamiltonian. As such, we will assume such property hold to facilitate the analysis. 
\end{enumerate}

Based on the structure of the problem \eqref{Ham_system}, one needs to find the numerical approximation for the gradient $J'_u$ or $H'_u$ and one challenge is finding numerical solution for the BSDE $(Y_t,Z_t)$. We first define the discrete control space:
\begin{align}
    \mathcal{U}^N:= \mathcal{U} \cap \mathcal{C}^N, \ \ \ \mathcal{C}^N:=  \big \lbrace u \big | u:= \sum^{N-1}_{i=1} a_n \mathbf{1}_{[t_n, t_{n+1}), n=0,...,N-1} \big | a_n \in \bR^k \big \rbrace    
\end{align}
that is, $\mathcal{U}^N$ consists of the piecewise control functions which are admissible. 

Classically, the BSDE finds the following discretization scheme, for some $u \in \mathcal{U}^N$ 
\begin{align}
Y^N_{t_n}&=\E_{t_n}[Y^N_{t_{n+1}}]+h\E_{t_n}[b'_x(t_n, X^N_{t_{n}},u_{t_{n}})Y^N_{t_{n+1}}+f'_x(t_n, X_n,u_{t_{n}})] \label{classical_y} \\
Z^N_{t_n}&=\E_{t_n}[Y^N_{t_{n+1}}\Delta W_{n}]/h \label{classical_z}
\end{align}
where $h:= \frac{T}{N}$ is the uniform mesh size, $\E_{t_n}[ \cdot ] := \E[ \cdot | \mathcal{F}_{t_n} ]$ and $X^N_{t_n}$ is the solution of the numerical SDE under the Euler scheme
\begin{align}
    X^N_{t_{n+1}}=X^N_{t_n}+b(t_n, X^N_{t_{n}},u_{t_n})h+ \sigma(t_n, u_{t_n}) \Delta W_{t_n} \ \ X_0=x_0. 
\end{align}
Typical methods include regression methods, Monte Carlo, finding values $(Y^N_{t_n},Z^N_{t_n})$ at fixed spacial locations and obtain function by interpolation, and deep learning \cite{Gobet1, Gobet2, pham4, QHan}. Then, the key update steps can be summarized as $\forall n \in \lbrace 0, ..., N-1 \rbrace$, $k=0,1,2...,K$:
\begin{align}\label{param_N_update}
\begin{cases}
 u_{t_n}^{N,k+1}= u_{t_n}^{N,k} - \eta_k \E[\nabla_{u} H(t_n, X^{N,u^k}_{t_n}, Y^{N,u^k}_{t_n}, Z^{N,u^k}_{t_n}, u^k_{t_n})],   \ \ \ &\text{Projection Algorithm}  \\
 u_{t_n}^{N,k+1} = (1-\rho) \max_{u \in U}\E[H(t_n,X^{N,u^k}_{t_n},Y^{N,u^k}_{t_n},Z^{N,u^k}_{t_n}, u)] + \rho u_{t_n}^k. \ \ \ &\text{Damped contraction}
\end{cases}
\end{align}

However, it is noted that solving BSDEs is overall a challenging and computationally expensive task. The challenge is even more pronounced under the projection or the current contraction framework as at each iteration $k \in \lbrace 1, ..., K \rbrace$ a new system for FBSDE has to be solved repeated. 

On the other hand, we note that the goal of solving the stochastic optimal control problem is to find either the value function or the control instead of solving explicitly the BSDEs. As such, we use (see also \cite{Hui1, feng1}) the sample-wise approximation of the BSDE which turns out to be an unbiased sample estimation of the classical numerical solution $(Y^N_{t_n},Z^N_{t_n})$. More specifically, we have: 
\begin{align}
X_{n+1} &=X_n + b(t_n,X_n, u_n) h + \sigma(t_n, u_n)\Delta W_n \label{x_sch} \\ 
Y_n &=Y_{n+1}+h \big( b_x(t_n,X_n, u_n) Y_{n+1} + f'_x(t_n,X_n, u_n) \big), \ \ Y_T = g_x(X_T)  \label{y_sch}  \\
Z_n &= \frac{Y_{n+1} \Delta W_n}{h} \label{z_sch} 
\end{align}
where $\Delta W_n \sim \mathcal{N}(0,h)$ is a Normal random variable.

For the projection algorithm, we observed that the simulated samples $Z_n$ typically has high variance with increasing number of temporal discretization $N$. To reduce variance, we simply use independently simulated trajectories of batch size $N$ or larger to compute the estimated gradient. 
That is, given $j'_{u}|_{t=t_n}:= \partial H_u(t_n,X^{u}_n,Y^{u}_n,Z^{u}_n, u_n)=b'_u(t_n,X^u_n,u_n)Y^u_n+\sigma'_u(t_n,X^u_n,u_n) Z^u_n +f'_u(t_n,X^u_n,u_n)$
we use $M$ in dependent samples to estimate the gradient. 
    \begin{align}
    \overline{(j'_{u})}_n&=\frac{1}{M}\sum^M_{i=1} \partial_u H (t_n,X^{u, i}_n,Y^{u, i}_n,Z^{u, i}_n, u_n ). 
\end{align}
It is expected that such batch estimator will lead to faster convergence. In the meantime, we note that an higher order convergence rate (1st order) is possible given a more accurate sampling for the solution of forward SDE than Euler is used. This is mostly because the temporal semi-discretization of solution of BSDE itself has first order convergence.  We will use $\Psi$ and 
\begin{align}
    X^{u}_{n+1} =\Psi(X^{u}_{n}, u^l_{n}, \Delta W_{n}),
\end{align}
to denote update for the state under a higher order scheme. We make explicit the choice for the higher order numerical scheme in the appendix. 

For the damped contraction algorithm, we still use the Euler discretization for the forward SDE while keeping the batch simulation for the solution of FBSDEs for the estimate of $u$ that minimizes the Hamiltonian. 

As such, based on the sampling scheme \eqref{x_sch}-\eqref{z_sch}, in light of \eqref{param_N_update}, we propose the Algorithm \ref{algorithm_batch_sample} for finding the optimal control. 

\begin{algorithm}
\caption{Algorithm for Batch sample-wise contraction via stochastic maximum principle.}\label{algorithm_batch_sample}
\begin{algorithmic}[1]
\REQUIRE 
Initializing the following 
\begin{itemize}
    \item The model parameters, i.e. functions $b,\sigma, f, g $, the batch size $M$, the total number of iterations $K$, $x_0$ and $\eta_k$ (learning rate) or $\rho$ the contraction rate.
    \item Total number of temporal discretization $N$, with terminal time $T=1$. 
    \item Initialize the control $u^0=0$ which is a vector of zeros. 
    \item Whether the update is based on Contraction or SGD. 
\end{itemize}
\FOR{$k=0,1,...,K-1$}
    \STATE{ 
\begin{enumerate}
    \item Simulate for $n=0,...,N-1$, $i \in \lbrace 1, ..., M \rbrace$,
    \begin{align}
        X^{u^k, i}_{n+1} &=X^{u^k, i}_n + b(t_n,X^{u^k, i}_n, u^k_n) h + \sigma(t_n, u^k_n)\Delta W^{k,i}_n
    \end{align}
    or simulate $X^{u^k, i}_{n+1}$ based on a higher order scheme $\Psi(X^{u^k,i}_{n}, u^k_{n}, \Delta W^{k,i}_{t_n})$.
    \item Simulate backward for for $n=N-1,...,0$, $i \in \lbrace 1, ..., M \rbrace$:
    \begin{align}
        Y^{u^k, i}_n &=Y^{u^k, i}_{n+1}+h \big( b_x(t_n,X^{u^k, i}_n, u^k_n) Y^k_{n+1}+ f_x(t_n,X^{u^k, i}_n, u^k_n) \big), \ \ Y_T = g_x(X^{u^k, i}_T)   \\
        Z^{u^k, i}_n &= \frac{Y^{u^k, i}_{n+1} \Delta W^{k,i}_n}{h} 
\end{align}
\item Compute for each  $n=0,...,N-1$
\begin{enumerate} [i)]
    \item If Contraction (see Assumption \ref{Contraction_setup_assumption} for the definition of $\bar{H}$) : 
    \begin{align}
    \tilde{u}^{k+1}_n&=\frac{1}{M}\sum^M_{i=1}\Bar{H}(t_n,X^{u^k, i}_n,Y^{u^k, i}_n,Z^{u^k, i}_n) \nonumber \\ 
     u_n^{k+1} &= (1-\rho) \tilde{u}^{k+1}_n+ \rho u_n^k. 
\end{align}
\item If SGD: 
    \begin{align}
    \overline{(j'_{u})}^{k+1}_n&=\frac{1}{M}\sum^M_{i=1} \partial H_u (t_n,X^{u^k, i}_n,Y^{u^k, i}_n,Z^{u^k, i}_n, u_n^k ) \nonumber \\ 
     u_n^{k+1} &= -\eta_k \overline{(j'_{u})}^{k+1}_n +  u_n^k. 
\end{align}
\end{enumerate}
\end{enumerate}
    }
\ENDFOR
\RETURN
The collection of controls $\lbrace u^{K}_{n} \rbrace_{n=0,...,N-1}$.
\end{algorithmic}
\end{algorithm}

\section{Proof of convergence}
In this section, we prove the convergence of the projection algorithm (1st order convergence in $h$) under the strong convexity assumption as well as the convergence of the damped contraction mapping algorithm under stronger assumption. 

\begin{assumption} \label{strong_convexity}
We make the following strong convexity assumption on the cost functional $J$: 
for some constant $\lambda>0$, and any two controls $v,w \in \mathcal{U}$:
\begin{equation}
	\langle J_u'(v)-J_u'(w),v-w \rangle \geq \lambda ||v-w||_2^2
\end{equation}
Equivalently, one also have the following inequality: 
\begin{align}
	J(v)-J(w) \geq \langle J'(w), v-w \rangle + \frac{\lambda}{2}||v-w||^2_2.
\end{align}
\end{assumption}

\begin{assumption}{\label{Contraction_setup_assumption}}
(Assumption for the damped contraction algorithm, a linear quadratic case.)
   \begin{enumerate}[I.]
    \item  We assume that the drift takes the form $b(x,u):=Ax + Bu+ C$ or the form $b(x,u):=A(u+1)x + Bu+ C$. We assume the diffusion term takes the form $\sigma(u) = D u +F $. The running cost is assumed to take the form $\frac{1}{2} \int^T_0 |u_t|^2 dt$. 
    \item  We assume that there is a function $\bar{H}:[0,T]\times \bR^d \times \bR^d \times \bR^d \rightarrow U$ such that: 
    \begin{align}
        \bar{H}(t,X_t,Y_t,Z_t)=\text{argmin}_{u \in U} H(t,X_t,Y_t,Z_t,u).
    \end{align}
    Based on the first point, the map $\bar{H}$ takes the following form for some $\alpha, \beta, \gamma$: 
\begin{align}
    \bar{H}(t,X_t,Y_t,Z_t)=\alpha(X_t) Y_t+\beta_t Z_t+ \gamma(X_t).
\end{align}
We then further assume that there exists a negative $\mu$ such that 
    \begin{align}
        \langle \int^T_0 \E[\bar{H}(t,X^{u}_t,Y^{u}_t,Z^{u}_t)]-\E[\bar{H}(t,X^{v}_t,Y^{v}_t,Z^{v}_t)], u_t-v_t \rangle dt  \leq \mu ||u-v||_2^2. \label{assumption_convexity_2}
    \end{align}
\end{enumerate} 
\end{assumption}
\begin{remark}
We take $H'_u=BY+DZ+u$ which means $\bar{H}=-BY-DZ$, and so $\alpha=-B, \beta=-D$. 
    We note that $\mathrm{II}$ is intimately related to the strong convexity assumption: 
    \begin{align}
        \langle J_u'(v)-J_u'(w),v-w \rangle \geq \lambda ||v-w||_2^2
    \end{align}
which in the current set up reads as 
\begin{align}
    &\int^T_0 \E \Big[ \langle BY^u_t+DZ^u_t+u_t - (BY^v_t+DZ^v_t+v_t), u_t-v_t  \rangle dt \Big] \geq \lambda ||u-v||_2^2 \nonumber\\
   & \Leftrightarrow \int^T_0 \E \Big[ \langle \alpha Y^u_t+ \beta Z^u_t - (\alpha Y^v_t+ \beta Z^v_t), u_t-v_t  \rangle dt \Big]  \leq  -(\lambda-1)||u-v||_2^2
\end{align}
Hence, if the factor $\lambda$ is larger than $1$, $\mu$ will take negative values. We also comment it is also viable to assume $\mu$ takes a small enough positive value in the argument of Theorem \ref{thm_contract_discrete}.  
\end{remark}

To proceed with the proof, we first list the following propositions which can be proved based on Assumption \ref{assumption 1}. Those results highlights the fact that under strong regularity assumptions on the coefficients of the differential equations, the solution of the SDEs/BSDEs is stable with respect to the controls. 


\begin{proposition}\label{continuous_Xbound}
    Under Assumption \ref{assumption 1}, we have that for $u,v \in \mathcal{U}$: 
\begin{align}
    \sup_{t \in [0,T]}\E[|X_t^u-X^v_t|^2] \leq ||u-v||^2.
\end{align}
\end{proposition}

\begin{proposition}\label{stab_x}
	Under the Assumption \ref{assumption 1}, Consider the two discrete SDEs with control $u,v \in \mathcal{U}^N$. 
	 Let $\xuv_n:=\xu_n-X^v_n$, then 
	\begin{equation}
		\max_{0 \leq n \leq N} \E[|\xuv_n|^2] \leq C||u-v||^2.
	\end{equation}
\end{proposition}

\begin{proposition}{\label{stab_yz}}
	Under the Assumption \ref{assumption 1}, let $\yuv_n:=Y^{N,u}_{t_n}-Y^{N,v}_{t_n}$,$\zuv_n:=Z^{N,u}_{t_n}-Z^{N,v}_{t_n}$,
	  we have the following estimates, for $u, v \in \mathcal{U}^N$:
	\begin{equation}
		\sup_{0\leq n\leq N}\E[|\yuv_n|^2]+h\sum^{N-1}_{n=0} \E[|\zuv_n|^2] \leq C||u-v||^2.
	\end{equation}
\end{proposition}

Now we state the proposition related to the damped contraction mapping setting. 
\begin{proposition} \label{continuous_YZbound}
    Under  Assumption \ref{assumption 1} and Assumption \ref{Contraction_setup_assumption} for any control $u,v \in \mathcal{U}$, 
\begin{align}
    \sup_{0 \leq t \leq T} \E[|Y_t^u-Y_t^v|^2] + \E[\int^T_0 |Z_t^u-Z_t^v|^2 dt ] \leq ||u-v||^2.
\end{align}
\end{proposition}

\begin{proposition}{\label{contraction_stab_yz}}
Let  Assumption \ref{assumption 1} and Assumption \ref{Contraction_setup_assumption} hold. Let $\bar{H}$ be defined as follows and $u,v \in \mathcal{U}$
    \begin{align}
    \bar{H}(t,X_t,Y_t,Z_t)=\alpha(X_t) Y_t+\beta_t Z_t+ \gamma(X_t).
\end{align}
Then we have for some $C > 0$,
\begin{align}
    \int^T_0 \big|\E[\bar{H}(t,X^{u}_t,Y^{u}_t,Z^{u}_t)]-\E[\bar{H}(t,X^{v}_t,Y^{v}_t,Z^{v}_t)] \big|^2 dt \leq C ||u-v||^2.
\end{align}
\end{proposition}

The following result is important in understanding the sample-wise approximation of the solution of the FBSDE (decoupled) which essentially says that the conditional expectation of the sample-wise simulated $(Y_n,Z_n)$ is the solution of the BSDE under the classical numerical scheme. As such, we are generating unbiased estimates for the classical numerical solution of the BSDE. 
\begin{proposition}\label{prop_exp_equiv}
	Given the control $u^N \in \mathcal{U}^N $, the following relationships hold.
\begin{align}
\E_{t_n}[Y_n]&=Y^N_{t_n} \\
\E_{t_n}[Z_n]&=Z^N_{t_n} 
\end{align}
and so $\E[Y_n]=\E[Y^N_{t_n}], \  \E[Z_n]=\E[Z^N_{t_n}]$.
\end{proposition}
\begin{proof}
	Consider $n=N-1$ and recall that $Y^N_N(\cdot)=Y_N(\cdot)=g'_x(\cdot)$,
We start with $Z_n$,
\begin{align}
	\E_{t_{N-1}}[Z_{N-1}]&=\E[Y_N\Delta W_{N-1}/h|\mathcal{F}_{t_{N-1}}] \nonumber \\
	&=\E[g'_x(X_N)\Delta W_{N-1}/h|\mathcal{F}_{t_{N-1}}] \nonumber \\
	&=\E_{t_{N-1}}[Y^N_{t_N}\Delta W_{N-1}]/h  \\ 
	&=Z^N_{t_{N-1}}.
\end{align}

For $Y_{N-1}$, since $Y^N_{t_N}=g_x'(X_N)$ the following relationship follows trivially. 
\begin{align}
	\E_{t_{N-1}}[Y_{N-1}]&=\E_{t_{N-1}}[ Y_{N}+hb'_x(t_{N-1},X_{N-1},u^N_{t_{N-1}}) Y_N +hf'_x(t_{N-1},X_{N-1},u^N_{t_{N-1}})] \nonumber \\
	&= \E_{t_{N-1}}[g'_x(X_N) + hb_x(t_{N-1},X_{N-1},u^N_{t_{N-1}})g'_x(X_N)+hf_x(t_{N-1},X_{N-1},u^N_{t_{N-1}})] \\
	&=Y^N_{t_{N-1}}.
\end{align}
where the last equality is by \eqref{classical_y}. The case $n=N-2$ follows similarly: 
for the $Z$ term, we have
\begin{align}
	\E_{t_{N-2}}[Z_{N-2}]&=\E_{t_{N-2}}[Y_{N-1}\Delta W_{N-2}/h] \nonumber \\
	&=\E_{t_{N-2}}[\E_{t_{N-1}}[Y_{N-1}]\Delta W_{N-2}/h] \nonumber \\
	&=\E_{t_{N-2}}[Y^N_{t_{N-1}}\Delta W_{N-2}]/h  \\ 
	&=Z^N_{t_{N-2}}.
\end{align}
And for the $Y$ term, we have: 
\begin{align}
	\E_{t_{N-2}}[Y_{N-2}]&=\E_{t_{N-2}}[ Y_{N-1}+hb'_x(t_{N-2},X_{N-2},u^N_{t_{N-2}}) Y_{N-1} +hf'_x(t_{N-2},X_{N-2},u^N_{t_{N-2}})] \nonumber \\
	&= \E_{t_{N-2}}[\E_{t_{N-1}}[ Y_{N-1}]+ hb'_x(t_{N-2},X_{N-2},u^N_{t_{N-2}})\E_{t_{N-1}}[ Y_{N-1}]+hf'_x(t_{N-2},X_{N-2},u^N_{t_{N-2}})] \nonumber \\
        &= \E_{t_{N-2}}[ Y^N_{t_{N-1}} + hb'_x(t_{N-2},X_{N-2},u^N_{t_{N-2}}) Y^N_{t_{N-1}}+hf'_x(t_{N-2},X_{N-2},u^N_{t_{N-2}})] \\
	&=Y^N_{t_{N-2}}.
\end{align}
Hence, the conclusion is proved by repeating such argument recursively until $n=0$. 
\end{proof}
\subsection{Proof of 1st order convergence rate, the projection algorithm}
In this section, we prove the convergence of the proposed batch gradient descent algorithm. 
We first note that when the forward process $X_t$ is exact, the semi-discretization scheme:
\begin{align}
\begin{cases}\label{1_order_exact}
X_{t_{n+1}}&=X_{t_n} + \int_{t_{n}}^{t_{n+1}} b(X_t, u_t) dt +\int_{t_{n}}^{t_{n+1}} \sigma(u_t) d W_t \\ 
\hat{Y}^{N}_{t_n}&=\E_{t_n} \big[ \hat{Y}^{N}_{t_{n+1}}+h \big( b'_x(X_{t_n}, u_{t_n}) \hat{Y}^{N}_{t_{n}}  + f'_x(X_{t_n}, u_{t_n}) \big) \big], \ \ \hat{Y}_T = g_x(X_T)  \\
\hat{Z}^{N}_{t_n}&= \E_{t_n} \big[\frac{\hat{Y}^{N}_{t_{n+1}}\Delta W_n}{h} \big]    
\end{cases}
\end{align}
will have first order error convergence rate in the following sense: 
\begin{equation}\label{semi_time}
    \max_{0 \leq n\leq N} \E[\sup_{t_n \leq t \leq t_{n+1}}|Y_t-\hat{Y}^N_{t_n}|^2] + \sum^{N-1}_{n=0} \E[\int^{t_{n+1}}_{t_n} |Z_t - \hat{Z}^N_{t_{n}}|^2 dt ] \leq C h^2 
\end{equation}
for some constant $C$ that depends only on the initial $X_0$ and the data $b,\sigma, f$ (under sufficient smoothness assumptions) see \cite{chessari}, Section 3.1 and \cite{Gobet} . 
Based on Algorithm \ref{algorithm_batch_sample}, we note that the forward process can be replaced with higher order discretization schemes for instance, see Table 7.1 in \cite{Weinan1}. 

Now, consider the following set of discretized (decoupled) FBSDE: 
\begin{align}
\begin{cases}\label{1_order_approx}
X^N_{t_{n+1}}&= \Psi(X^N_{t_n}, u_{t_n}, \Delta W_{t_n}) \\ 
Y^{N}_{t_n}&=\E_{t_n} \big[ Y^{N}_{t_{n+1}}+h \big(  b'_x(X^N_{t_n}, u_{t_n}) Y^{N}_{t_{n}}  + f'_x(X^N_{t_n}, u_{t_n}) \big) \big], \ \ Y^N_T = g_x(X^N_T)  \\
Z^{N}_{t_n}&= \E_{t_n} \big[\frac{Y^{N}_{t_{n+1}}\Delta W_{t_n}}{h} \big]  
\end{cases}
\end{align}
where $\Psi: \bR^d \times \mathbb{R}^k \times \mathbb{R}^b \rightarrow \bR^d$ is a higher order numerical scheme for $X_t$ of order $\Theta$, we want to study the difference between $(\hat{Y}^{N}_{t_n}, \hat{Z}^{N}_{t_n})$ and $(Y^{N}_{t_n},Z^{N}_{t_n})$.  We note that in both \eqref{1_order_exact} and \eqref{1_order_approx}, implicit form is used for simplicity of analysis. We comment that both the implicit and explicit schemes have the same order of convergence when $h$ is taken small enough, see also Section 5.3, Remark 5.3.2 in \cite{Jianfeng}.  
\begin{lemma}\label{first_order_FBSDE}
Let $\Psi$ defined above be a scheme of order $\Theta$, assume that there exists constant $C$ such that $\sqrt{\E[|X^N_{t_n}-X_{t_n}|^4]} \leq C \E[|X^N_{t_n}-X_{t_n}|^2]$, then we have
\begin{equation}\label{high_order_discretization}
    \max_{0 \leq n\leq N} \E[\sup_{t_n \leq t \leq t_{n+1}}|Y_t-Y^N_{t_n}|^2] + \sum^{N-1}_{n=0} \E[\int^{t_{n+1}}_{t_n} |Z_t - Z^N_{t_{n}}|^2 dt ] \leq C h^{2}.
\end{equation}
provided $2\Theta \geq 3$. 
\end{lemma}
\begin{proof}
    By definition of $\hat{Y}^{N}_{t_n}$ and $Y^{N}_{t_n}$, we have that by the Martingale Representation theorem, there exists square integrable process $\hat{\tilde{Z}}_s,\tilde{Z}_s$ such that 
    \begin{align}\label{mart_part}
        \begin{cases}
           \hat{Y}^{N}_{t_{n+1}} &=\hat{Y}^{N}_{t_n}- h \partial_{x} H(t_n, X_{t_n}, \hat{Y}^N_{t_{n}},\hat{Z}^N_{t_{n}}, u_{t_n} ) + \int^{t_{n+1}}_{t_n} \hat{\tilde{Z}}_s dW_s \nonumber\\ 
            Y^{N}_{t_{n+1}}&=Y^{N}_{t_n}-h \partial_{x} H(t_n, X^N_{t_n}, Y^N_{t_{n}},Z^N_{t_{n}} , u_{t_n} )+\int^{t_{n+1}}_{t_n} \tilde{Z}_s dW_s \\
        \end{cases}
    \end{align}
We take the difference between the two equations above and denote $\Delta Y^N_{t_{n+1}}:=\hat{Y}^{N}_{t_{n+1}}-Y^{N}_{t_{n+1}}$, $\Delta Z^N_{t_{n+1}}:=\hat{Z}^{N}_{t_{n+1}}-Z^{N}_{t_{n+1}}$ and $\Delta \tilde{Z}_{s}:=\tilde{\hat{Z}}_{s}-\tilde{Z}_{s}$, we get: 
\begin{align}
    \Delta Y^N_{t_{n}} + \int^{t_{n+1}}_{t_n} \Delta \tilde{Z}_{s} dW_s = Y^N_{t_{n+1}} + h \Delta \partial_x H 
\end{align}
square both sides and take expectation, note the independence between $\Delta Y^N_{t_{n}}$ and $Z^N_{s}$ we have: 
\begin{align}
    \E \Big[ |\Delta Y^N_{t_{n}}|^2 + \int^{t_{n+1}}_{t_n} |\Delta \tilde{Z}_{s}|^2 ds \Big] & \leq (1+h)\E[|Y^N_{t_{n+1}}|^2]+ (1+\frac{1}{h})\E[ |\Delta \partial_x H|^2] h^2
\end{align}
Note that 
\begin{align}
    \E[ |\Delta \partial_x H|^2]& = \E \Big[ \Big| b'_x(X^N_{t_n},u_{t_n})Y^N_{t_{n}}-b'_x(X^N_{t_n},u_{t_n})\hat{Y}^N_{t_{n}}+b'_x(X^N_{t_n},u_{t_n})\hat{Y}^N_{t_{n}}-b'_x(X_{t_n},u_{t_n})\hat{Y}^N_{t_{n}}  \nonumber\\ 
    & + f'_x(X^N_{t_n},u_{t_n}) - f'_x(X_{t_n},u_{t_n}) \Big|^2 \Big] \nonumber\\
    & \leq C\E[|\Delta Y^N_{t_{n}}|^2] + C\E[|\hat{Y}^N_{t_{n}}|^2|X^N_{t_n}-X_{t_n}|^2] + C \E[|X^N_{t_n}-X_{t_n}|^2] \nonumber \\ 
    & \leq C\E[|\Delta Y^N_{t_{n}}|^2] + C\sqrt{\E[|\hat{Y}^N_{t_{n}}|^4] \E[|X^N_{t_n}-X_{t_n}|^4]} + C \E[|X^N_{t_n}-X_{t_n}|^2] \nonumber \\ 
    & \leq C\E[|\Delta Y^N_{t_{n}}|^2] + C h^{2\Theta} 
\end{align}
Then, we have that for $h$ small enough, 
\begin{align}\label{square_compare}
    \E \Big[ |\Delta Y^N_{t_{n}}|^2 + \int^{t_{n+1}}_{t_n} |\Delta \tilde{Z}_{s}|^2 ds \Big] & \leq (1+Ch)\E[|Y^N_{t_{n+1}}|^2]+ C h^{2\Theta} h
\end{align}
By discrete Gronwall's inequality,  
\begin{align}\label{Y_square_compare}
    \sup_{0 \leq n \leq N} \E[|\Delta Y^N_{t_n}|^2] \leq C h^{2\Theta} .
\end{align}
In the meantime, in \eqref{mart_part}, we multiply both sides by $\Delta W_{t_n}$, take expectation and use It\^{o} isometry
\begin{align}
\begin{cases}
    \hat{Z}^N_{t_n} &= \E_{t_n}[ \int^{t_{n+1}}_{t_n} \hat{\tilde{Z}}_s ds ]/ h \nonumber \\
    Z^N_{t_n} &= \E_{t_n}[ \int^{t_{n+1}}_{t_n} \tilde{Z}_s ds] / h 
\end{cases}
\end{align}
We then take difference between the two equations above, square both sides and then take expectation to get 
\begin{align}
    \E[|\hat{Z}^N_{t_n}-Z^N_{t_n}|^2] & \leq \frac{1}{h}\E[ \int^{t_{n+1}}_{t_n} |\Delta \tilde{Z}_s|^2 ds] \leq h^{2\Theta -1}
\end{align}
where the last inequality comes from \eqref{square_compare} and \eqref{Y_square_compare}. 
This combined with \eqref{semi_time} proves \eqref{high_order_discretization}.  
\end{proof}

\begin{lemma}{\label{ea1}} 
Under Assumptions \ref{assumption 1}, and that $Y^{u^N}_t, Z^{u^N}_{t} \in C_b^{1,2}$, 
then the following holds: there exist some constant $C \geq 0$ such that for 
\begin{equation}
	\epsilon^2_N:=\sup_{u^N \in \mathcal{U}^N} ||J'(u^N)-J'_N(u^N)||_2^2 
\end{equation}
we have  
	\begin{equation}
	\epsilon^2_N  \leq  C h^2. 
\end{equation}
\end{lemma}
\begin{proof}
	To simplify the notations, we make the following definitions: 
	\begin{align}
	\phi |_{t=t_n}&= f'_u(X_{t_n},u^N_{t_n})^T Y_{t_n}+g'_u(X_{t_n},u^N_{t_n})^T Z_{t_n}+r'_u(X_{t_n},u^N_{t_n})\\
		\phi_n 		& = f'_u(X^N_{t_n},u^N_{t_n})^TY^{N}_{t_n}+g'_u(X^N_{t_n},u^N_{t_n})^TZ^{N}_{t_n} +r'_u(X^N_{t_n},u^N_{t_n})
	\end{align}
	That is, $\phi_n$ is the discrete dynamics where all the components are obtained from discretization; whereas $\phi_{t_n}$ results from the continuous dynamics using the control $u^N$ which is piecewise constant.
	Hence, by defining $\bar{\phi}(t,x_t):=\phi(t,x_t)$ on the interval $[t_n,t_{n+1})$,
	we then have: 
	\begin{align}
		\int^T_0 \ (J'(u^N)\big|_{t}  - J'_N(u^N) \big|_{t}  \ )^2 dt 
		&\leq2\sum^{N-1}_{n=0} \int^{t_{n+1}}_{t_n} (J'(u^N)|_t - J'(u^N)|_{t_n})^2+ (J'(u^N)|_{t_n}-J_N'(u^N)|_{t_n})^2 dt \nonumber \\ 
		 &= 2\int_0^T \E[ \phi_t-\phi_{t_n}]^2 + \E[\phi_{t_n} - \phi_{n} ]^2 dt  \nonumber \\
		& \leq C \sum^{N-1}_{n=0} \int^{t_{n+1}}_{t_n} (\int^t_{t_n}  \frac{d}{dr} \E[\bar{\phi}_r] \Big \rvert_{r=s} ds  )^2 dt +C h \sum^{N-1}_{n=0} (\E[\phi_{t_n}]-\E[\phi_{n}])^2 \label{error_part} \\
		& \leq \mathcal{O}(h^2)\label{error_s1}
\end{align}
where from \eqref{error_part} to \eqref{error_s1}, we used the fact that $\frac{d}{dr} \E[\bar{\phi}_r] \Big \rvert_{r=s}$ is uniformly bounded, the higher order error between the solution of SDEs and its numerical approximation and Lemma \ref{first_order_FBSDE}.
\end{proof}

We also show an estimate on the gradient of the loss function. Assume $u$ defined on $[0,T]$ is continuous. 
Let $\bar{u}^N = \text{argmin}_{u \in \mathcal{U}^N} ||u-u^*||_2$ which is the projection of the true optimal control onto the space $\mathcal{U}^N$. We have, by recalling $J'(u^*)=0$, the following set of inequalities: 
\begin{align}
    J(\bar{u}^N)-J(u^*) &= \int^1_0 \frac{d}{d \epsilon} J(u^* + \epsilon (\bar{u}^N- u^*))d \epsilon \ \ \nonumber\\
    &=  \int^1_0 \langle J'_u(u^* + \epsilon (\bar{u}^N- u^*)) - J'(u^*),\bar{u}^N- u^* \rangle   d \epsilon \nonumber \\
    & \leq \int^1_0 ||\bar{u}^N- u^*||_2 ||\bar{u}^N- u^*|| d \epsilon \nonumber\\
    & \leq C||\bar{u}^N- u^*||^2_2 \leq C \frac{1}{N^2}
\end{align}
In the meantime, let $u^{*,N} \in \mathcal{U}^N$ be the optimal control of the problem in the space $\mathcal{U}^N$, we then have
\begin{align}
    J(u^{*,N})-J(u^*) & \leq J(\bar{u}^N)-J(u^*) \nonumber\\
    &\leq C \frac{1}{N^2}
\end{align}
By the strong convexity assumption, we have by using $J'(u^*)=0$,
\begin{align}
    \langle J'(u^*), u^{*,N} -u^* \rangle + \frac{\lambda}{2} ||u^{*,N} -u^*||^2 &\leq  J(u^{*,N})-J(u^*) \nonumber \\ 
    &\Rightarrow  ||u^{*,N} -u^*||^2 \leq C \frac{1}{N^2}. \label{projection_ustar_estimate}
\end{align}
We then further have 
\begin{align}
||J'(u^{*,N})-J'(u^*)||^2 &\leq C ||u^{*,N}-u^{*}||^2 \nonumber \\
&\Rightarrow  ||J'(u^{*,N})||^2 \leq C \frac{1}{N^2}
\end{align}
We now define the information set to be the following: 
we let $\mathcal{G}_{k-1}:=\sigma\Big( \Big \lbrace u_0, \lbrace W^i_n,  X^i_n,Y^i_n,Z^i_n\rbrace^{N}_{n}  \Big  \rbrace^{k-1,M}_{l=0,i=1}\Big)$, that is the $\sigma-$algebra generated by the initial control $u_0$ and the $M$ independent paths $\lbrace X^i_n,Y^i_n,Z^i_n \rbrace^{M,N-1}_{i,n}$. Also note that $u^k$ is $\mathcal{G}_{k-1}$ measurable. We use $\E^K[ \cdot ]$ to denote $\E[\cdot | \mathcal{G}_K]$.

\begin{theorem}{\label{standard_new}}
Assume that Assumption \ref{assumption 1} and Assumption \ref{strong_convexity} hold. 
Take $\eta_k= \theta/(k+M)$ for some $\theta,M$ such that $\lambda \theta - 4C_L\theta^2/(1+M)>2$. Let $\lbrace u^k \rbrace$ be defined as in Algorithm \ref{algorithm_batch_sample}, 3-(ii).  Then for large K, the following inequality holds: 
	\begin{equation} \label{model_error_cvx}
		\E[|| u^{K+1}-u^{*,N}||^2_2] \leq C ( \frac{1}{K} + N^{-2} )
	\end{equation}
\end{theorem}
\begin{remark}
    The main idea of the proof is similar to that of Theorem 3.8 from \cite{Hui1}, and we present the major argument where the convergence rate was improved. 
\end{remark}
\begin{proof}
    Consider the following two equations:
\begin{align}
	u^{*,N} & =\mathcal{P}_{K_N} (u^{*,N} -\eta_k  J'(u^{*,N}) +\eta_k  J'(u^{*,N}) \label{opt1}) \\
		u^{K+1} &= \mathcal{P}_{K_N}(u^K- \eta_k \overline{j}'(u^K)) \label{opt2}
	\end{align}
    Take the difference, square both sides, use Young's inequality with $\epsilon$ and take conditional expectation: 
    \begin{align}
        \E^K[|| u^{K+1}-u^{*,N}||^2_2] ]&\leq (1+\epsilon ) \E^K[ ||(u^K-u^{*,N})-\eta_K (\bar{j}'(u^{K})-J'(u^{*,N})) ||^2 ]  + (1+1/\epsilon ) \eta_K^2 \E^K[||J'(u^{*,N})||^2] \nonumber \\
       &\leq (1+\epsilon ) \Big( \E^K[||u^{K} -u^{*,N}||^2_2]-2 \eta_K\langle \E^K[\overline{j}_N'(u^{K})]-\E^K[J'(u^{*,N}) ] ,u^{K} -u^{*,N}\rangle  \\
	&+ \eta_K^2 \E^K[||\overline{j}_N'(u^{K})-J'_N(u^K) +J'_N(u^K)- J'(u^{*,N})||^2_2] \Big) + (1+1/\epsilon )\eta_K^2 ||J'(u^{*,N})||^2 \nonumber \\
    &\leq (1+\epsilon ) \Big( \E^K  ||u^{K}-u^{*,N}||^2_2 - 2 \eta_K \E^K \Big[ \langle J_N'(u^{K})-J'(u^{K})+J'(u^{K})-J'(u^{*,N}), u^{K} -u^{*,N} \rangle \Big] \nonumber \\ 
	& + 2\eta_K^2( \E^K[ ||J_N'(u^{K})-J'(u^{*,N})||_2^2 ]+ \underbrace{||\overline{j}_N'(u^{K})-J'_N(u^K)||^2}_{\mathlarger{*}} )\Big) + (1+\frac{1}{\epsilon})\eta_K^2C/N^2 \label{intermediate_1}
    \end{align}
We note the following inequalities: 
\begin{align}
-2\eta_K \langle J_N'(u^{K})-J'(u^{K}), u^{K} -u^{*,N} \rangle & \leq \eta_K \bigg( \frac{||J'_N(u^K)-J'(u^K)||^2}{\lambda} + \lambda ||u^K-u^{*,N} ||^2_2 \bigg) \nonumber \\ 
	&\leq \frac{\eta_K}{\lambda} \epsilon_N^2  + \lambda \eta_K||u^K-u^{*,N} ||^2_2 
\end{align}
Also, we have by the strong convexity assumption: 
\begin{align}
    2 \eta_K \langle J'(u^{K})-J'(u^{*,N}), u^{K} -u^{*,N} \rangle \geq 2 \lambda \eta_K ||u^K-u^{*,N}||^2_2
\end{align}
Next, we have
\begin{align}
    2\eta_K^2 \E^K[ ||J_N'(u^{K})-J'(u^{*,N})||_2^2 ] & \leq 4\eta_K^2 \E^K[ ||J_N'(u^{K})-J'(u^K)||^2 + || J'(u^K)-J'(u^{*,N})||_2^2 ] \nonumber\\ 
    & \leq 4 C_L\eta^2_N \epsilon^2_K+4 C \eta^2_K ||u^K-u^{*,N}||^2
\end{align}
and that 
\begin{align}
    \mathlarger{*} &= \Big|\Big|\E\Big[b'_u(X^N_{t_n}, u_{t_n})Y^N_{t_n} + \sigma'_u(u_{t_n})Z^N_{t_n} +f'_u(X^N_{t_n}, u_{t_n}) -\frac{1}{B}\sum^B_{l=1} b'_u(X^{N,l}_{t_n}, u_{t_n})Y^l_{n} + \sigma'_u(u_{t_n})Z^l_{n} +f'_u(X^{N,l}_{t_n}, u_{t_n}) \Big]\Big|\Big|^2 \nonumber\\
    & \leq 6C(\frac{1}{N^2}+1)+\frac{1}{B^2} \sum^{B}_{i,j} \E^K \big[ \E[\langle Z^N_{t_n}-Z^i_{n}, Z^N_{t_n}-Z^j_{n} \rangle | \mathcal{G}_K \bigvee \mathcal{F}_{\lbrace X_{0,...,n}^{u^k,i} \rbrace^M_{i=1}}  ] \big] \nonumber \\ 
    &\leq C+ \frac{1}{B^2} \sum^{B}_{i} \E^K \big[|Z^N_{t_n}-Z^i_{n}|^2]  \leq C (1+ \frac{N}{B}),  \label{avg_Z_estimate}
\end{align}
where in the last inequality, we used Proposition \ref{prop_exp_equiv} and Lemma 3.3 in \cite{Hui1}. That is $\sup_{0\leq n \leq N} \E[|Y_n|^2] \leq C,\sup_{0\leq n \leq N-1} \E[|Z_n|^2] \leq CN$ uniformly for all control $u \in \mathcal{U}^N$.  When we pick $B=N$, we obtain $\mathlarger{*} \leq C$. 
As such, putting everything together we have 
\begin{align}
    \eqref{intermediate_1} &\leq (1+\epsilon) \E^K \Big[ ||u^K-u^{*,N}||^2_2 +\frac{\eta_K}{\lambda} \epsilon_N^2  + \lambda \eta_K||u^K-u^{*,N} ||^2_2- 2 \lambda \eta_K ||u^K-u^{*,N}||^2_2 \nonumber \\ 
	&+4 \eta_K^2C_L ||u^K-u^{*,N}|| + 4 \eta_K^2 \epsilon^2_N+2C \eta_K^2 N \Big]  + (1+\frac{1}{\epsilon})(\eta_K^2C/N^2) \nonumber \\ 
	& \leq  (1+\epsilon) \Big((1- \lambda \eta_K +4 \eta_K^2C_L) \E^K \big[ ||u^K-u^{*,N}||^2_2 \big] + C(\eta_K + \eta_K^2) \epsilon_N^2 + 2C \eta_K^2 \Big) + (1+\frac{1}{\epsilon})(\eta_K^2C/N^2) \nonumber \\ 
    & \leq (1+\epsilon ) \Big( (1- \tilde{c} \eta_K \big)\ \E^K[ ||u^{K} -u^{*,N}||^2_2 ]+ C\eta_K \epsilon_N^2 + 2C\eta_K^2 \Big)+(1+\frac{1}{\epsilon}) \eta^{2}_KC/N^2 
\end{align}
where, $\tilde{c}=\lambda-4C_L \eta_k$. 
By taking $\eta_K \sim \frac{\theta}{K+M}$, let $\eta'_K=1/(K+M)$, picking $M$ and $\theta$ such that 
\begin{equation}\label{c_larger_than1}
	2c:=\lambda \theta - 4C_L\theta^2/(1+M)>2
\end{equation}
then we have for some constant $C$,
\begin{equation} \label{base_1}
	\E^K[|| u^{K+1}-u^{*,N}||^2_2] \leq (1+\epsilon ) \Big( (1-2c \eta'_K ) \E^K[ ||u^{K} -u^{*,N}||^2_2] +2C \eta'^2_K+ C\eta'_K \epsilon_N^2 \Big)+(1+\frac{1}{\epsilon})\eta'^2_K \theta^2 C/N^2
\end{equation}
By picking $\epsilon=c \eta'_K$, we have that for some constant $C>0$
\begin{align}\label{base}
    \eqref{base_1} \leq &  (1-c \eta_K' )\E^K[ ||u^{K} -u^{*,N}||^2_2]+ 2C\eta'^2_K+C \eta'_K \epsilon^2_N
\end{align}
Then, by using \eqref{base}, we have that for some $M>0$: 
\begin{align}
    \E[|| u^{K+1}-u^{*,N}||^2_2] \leq (K+M)^{-c}||u^0- u^{*,N}||^2_2 + C (K+M)^{-1} + CN^{-2} 
\end{align}
By using the fact that $c>1$, we conclude that for some $C >0$, 
\begin{align}
    \E[|| u^{K+1}-u^{*,N}||^2_2] \leq C ( \frac{1}{K} + N^{-2} )
\end{align}
\end{proof}
We now state the first main theorem as follows. 
\begin{theorem}\label{main_projection}
    Let the assumption made in Theorem \ref{standard_new} hold, then one has 
\begin{align}
\E[||u^* - u^{K+1}||^2] \leq C(\frac{1}{K} + \frac{1}{N^2}).  
\end{align}
\end{theorem}
\begin{proof}
    From \eqref{projection_ustar_estimate}, we have 
    \begin{align}
        \E[||u^* - u^{K+1}||^2] &\leq 2 \E[||u^* - u^{N,*}||^2] + 2 \E[||u^{N,*}-u^{K+1}||^2] \nonumber\\
        &\leq C(\frac{1}{K} + \frac{1}{N^2}). 
    \end{align}
    which concludes the proof. 
\end{proof}

\subsection{Proof of contraction algorithm, a special case}
We now proceed to show convergence of Algorithm \ref{algorithm_batch_sample} for the damped contraction case and we focus on a special linear quadratic case of the control problem. 

\begin{theorem}\label{thm_contract_discrete}
    Under Assumptions \ref{assumption 1}-\ref{Contraction_setup_assumption},
    there exists $\rho \in (0,1)$, $\eta: (0,1) \rightarrow (0,1)$ and $\eta(\rho) <1$ so that: 
   \begin{align}
    \E^{k}[||u^{N,*}-u^{k+1}||^2] \leq C ( \eta^{k+1} + C\frac{N}{M}+ \frac{1}{N}).
\end{align}
where $u^k$ is defined via Algorithm \ref{algorithm_batch_sample} 3-i). 
\end{theorem}
\begin{proof}
We introduce the following shorthand notation:  
\begin{align}
    \bar{H}_t^{',u}  &= \E^u[ \bar{H}(t,X^{u}_t,Y^{u}_t,Z^{u}_t) ] \label{symb_note1} \\ 
    \bar{H}_n^{u}    &=  \bar{H}(t_n,X^{N, u}_n,Y^{u}_n,Z^{u}_n) \label{symb_note2} \\ 
    \bar{H}_n^{N,u}  &=  \bar{H}(t_n,X^{N,u}_{t_n},Y^{N,u}_{t_n},Z^{N,u}_{t_n}) \label{symb_note3} \\
    \bar{H}_n^{',N,u}  &=  \E^u[\bar{H}(t_n,X^{N,u}_{t_n},Y^{N,u}_{t_n},Z^{N,u}_{t_n}) ] \label{symb_note4}
\end{align}
and we comment that under the Euler scheme, we use the notation $X^{N,u}_n, X^{N,u}_{t_n}$ interchangeably. The upper index $N$ simply means it is under temporal numerical discretization.

We note that given the same control $u \in \mathcal{U}^N$, \eqref{symb_note2} and \eqref{symb_note3} have the same expectation. To wit:
\begin{align}
    \E^u[\bar{H}_n^{u}]&= \E^u[\alpha(X^{N,u}_n) Y^u_n+\beta Z^u_n+ \gamma(X^{N,u}_n)] \nonumber \\ 
    &= \E^u[\alpha(X^u_n) \E^u[Y^u_n|\mathcal{F}^u_{t_n}]+\beta \E^u[Z^u_n|\mathcal{F}^u_{t_n}]+ \gamma(X^{N, u}_n)] \nonumber \\ 
    &= \E^u[\alpha(X^{N,u}_n) Y^{N,u}_{t_n}+\beta Z^{N,u}_{t_n}+ \gamma(X^{N, u}_{n})] \nonumber \\ 
    &= \bar{H}_n^{',N,u},
\end{align}
where we used $\E^u$ to emphasize the fact that the expectation is taken given control $u$.  Also, in the second equality above, we used the linear structure of $\bar{H}_n^{u}$ and the measurability of $\alpha, \beta$. 

Consider the following two sets of equations: 
\begin{align}
        u_{t}^{N,*}|_{t=t_n} &= \mathcal{P}_N \Big( (1-\rho)\E[\Bar{H}(t,X^{u^{N,*}}_{t},Y^{u^{N,*}}_{t},Z^{u^{N,*}}_{t})] + \rho u_{t}^{N,*} \Big) \Big|_{t=t_n} \label{disc_diff_1} \\ 
        u_n^{k+1} &= \mathcal{P}_N \Big((1-\rho) \frac{1}{M}\sum_{i=1}^{M}\Bar{H}(t_n,X^{u^k,i}_n,Y^{u^k,i}_n,Z^{u^k,i}_n) + \rho u_n^k \Big) \label{disc_diff_2}
\end{align}
Take the difference between \eqref{disc_diff_1} and \eqref{disc_diff_2} and take the $L_2$ norm to obtain: 
\begin{align}
    ||u^{N,*}-u^{k+1}||^2 &\leq  (1-\rho)^2  \int^T_0 \Big( \frac{1}{M}\sum_{i=1}^{M} (\bar{H}_n^{u^k,i} - \bar{H}_{t}^{',u^{N,*}}) \Big)^2 dt + 2(1-\rho)\rho \underbrace{\int^T_0 \langle\frac{1}{M}\sum_{i=1}^{M} (-\bar{H}_n^{u^k,i} + \bar{H}_t^{',u^{N,*}}), u_{t}^{N,*}-u_n^{k} \rangle dt}_{\mathlarger{B_0}}  \nonumber \\ 
    &+ \int^T_0 \rho^2 |u_{t_n}^{N,*}-u_n^{k}|^2  dt \label{diff_samplewise_1}
\end{align}
For the first term in \eqref{diff_samplewise_1}, we have by taking expectation on both sides: 
\begin{align}
 \int^T_0 \E^{k} \Big [\Big( \frac{1}{M}\sum_{i=1}^{M} (\bar{H}_n^{u^k,i} - \bar{H}_n^{',u^{N,*}}) \Big)^2 \Big] dt &= \underbrace{ \int^T_0 \frac{1}{M^2}\sum^M_{i=1} \E^{k} \Big [ (\bar{H}_n^{u^k,i} - \bar{H}_t^{',u^{N,*} })^2 \Big ] dt}_{\mathlarger{B_1}}  \nonumber\\ 
&+ \frac{1}{M^2}\sum_{i\neq j}^{M}  \underbrace{ \int^T_0 \E^{k} \Big [ (\bar{H}_n^{u^k,i} - \bar{H}_t^{',u^{N,*}}) (\bar{H}_n^{u^k,j} - \bar{H}_t^{',u^{N,*}}) \Big ] dt }_{\mathlarger{B_2}}
\end{align}
For the term $B_1$, we have by standard estimates that $\int^T_0 \E^k[|\bar{H}_n^{u^k,i}|^2] dt \leq CN$  and $\int^T_0 |\bar{H}_t^{',u^{N,*}}|^2 dt \leq C$ where the first `$N$' term originates from the bound in $Z^{u^k}_n$. 
Thus, the term $B_1$ is uniformly bounded by $C \frac{N}{M}$. For the second term: 
\begin{align}
B_2&= \frac{M^2-M}{M^2} \int^T_0 |\bar{H}_n^{',N, u^k}-\bar{H}_t^{',u^{N,*}}|^2 dt \nonumber \\
& \leq \int^T_0 |\bar{H}_n^{',N, u^k}-\bar{H}_{t_n}^{',u^k} + \bar{H}_{t_n}^{',u^k}-\bar{H}_{t_n}^{',u^{N,*}}+\bar{H}_{t_n}^{',u^{N,*}}-\bar{H}_{t}^{',u^{N,*}}|^2 dt \nonumber \\ 
& \leq 3 \int^T_0 |\bar{H}_n^{',N, u^k}-\bar{H}_{t_n}^{',u^k}|^2 dt + C ||u^{N,*}-u^k||^2 + C \frac{1}{N^2} \nonumber\\
& \leq C \frac{1}{N} + C ||u^{N,*}-u^k||^2,
\end{align}
where we have used independence among samples and that: 
\begin{align}
\E^k[\bar{H}_n^{u^k}-\bar{H}_n^{',u^{N,*}}] &=\E^{k} \Big[\E^k[\bar{H}_n^{u^k}]-\bar{H}_{t_n}^{',u^{N,*}} \Big] \nonumber \\
&= \E^{k} \Big[ \bar{H}_n^{',N,u^k} -\bar{H}_{t_n}^{',u^{N,*}} \Big].
\end{align}
And in the second to last inequality, we also used Proposition \ref{contraction_stab_yz}.

As such we have:
\begin{align}
    \sum_n  \E \Big [\Big( \frac{1}{M}\sum_{i=1}^{M} (\bar{H}_n^{u^k,i} - \bar{H}_n^{',u^{N,*}}) \Big)^2 \Big] h & \leq C(\frac{1}{M} + \frac{1}{N}) + C ||u^k-u^{N,*}||^2.
\end{align}

In the meantime, we also have:
\begin{align}
    2\E^{k}[\mathlarger{B_0}] & = 2 \int^T_0 \langle -\bar{H}_n^{',N,u^k,i} + \bar{H}_t^{',u^{N,*}}, u_{t_n}^{N,*}-u_n^{k} \rangle dt \nonumber \\ 
    & = 2 \int^T_0 \langle -\bar{H}_n^{',N,u^k,i} + \bar{H}_{t}^{',u^{k}},u_{t_n}^{N,*}-u_n^{k} \rangle+ \langle  -\bar{H}_{t}^{',u^k} + \bar{H}_t^{',u^{N,*}}, u_{t}^{N,*}-u_n^{k} \rangle dt \nonumber\\
    & \leq  2|\mu|||u^{N,*}-u^{k}||^2 + \frac{1}{2|\mu|} \int^T_0 | \bar{H}_{n}^{',N,u^k,i}- \bar{H}_t^{',u^{k}}|^2 dt + 2\mu ||u^{N,*}-u^{k}||^2 \label{cancellation} \\
    & \leq C \frac{1}{N}
\end{align}
where in the last inequality, we used the assumption that $\mu$ is negative and in $\eqref{cancellation}$ we used assumption made in the theorem. 
Putting everything together, we have
\begin{align}
    \E^{k}[||u_{t_n}^{N,*}-u_n^{k+1}||^2] & \leq (1-\rho)^2 C \Big( \frac{1}{M}+ \frac{1}{N} + ||u^k-u^{N,*}|| \Big) + (1-\rho)\rho C \frac{1}{N} + \rho^2 ||u^k-u^{N,*}||^2 \nonumber \\ 
    & \leq  (C (1-\rho)^2 + \rho^2)||u^k-u^{N,*}||^2 + C (\frac{1}{M}+ \frac{1}{N}) \label{contraction_ineq_1}
\end{align}
Note that in \eqref{contraction_ineq_1}, there exists $\rho \in (0,1)$ such that $\eta:=(C (1-\rho)^2 + \rho^2) <1$. For instance, this can be achieved by taking $\rho=\frac{C}{C+1}$. Thus, by iterating $\eqref{contraction_ineq_1}$, we then have: 
\begin{align}
    \E^{k}[||u^{N,*}-u^{k+1}||^2]&\leq \eta^{k+1} ||u^{N,*}-u^{0}||^2 + \sum^k_{l=0} \eta^l C(\frac{N}{M}+ \frac{1}{N}) \nonumber\\
    & \leq C ( \eta^{k+1} + \frac{N}{M}+ \frac{1}{N})
\end{align}
\end{proof}
\begin{theorem}
    Let the assumption made in \ref{Contraction_setup_assumption} hold, then one has for some appropriate $\rho \in (0,1)$, there exist $\eta \in (0,1)$ such that 
\begin{align}
    \E[||u^{N,*}-u^{k+1}||^2] \leq C ( \eta^{k+1} + \frac{N}{M}+ \frac{1}{N}).
\end{align}
\end{theorem}
The proof is similar to that for Theorem \ref{main_projection}. 

\section{Numerical Example}
In this section, we demonstrate the convergence of the algorithm by studying a few numerical examples. We show examples for both cases where the diffusion are controlled/uncontrolled. More specifically: 
\begin{enumerate}
    \item For the Batch projection algorithm with high order in forward SDE, we note that both numerical examples agree with the analysis: a first decay is observed when the number of iterations is taken as $K \sim \mathcal{O}(N^2)$. 
    \item For the damped contraction algorithm, we examine the convergence behavior of the algorithm by choosing different hyperparameter for $K,N,M$. We comment that to obtain exactly the rate of convergence that match the experiment results, a more refined analysis is needed. 
    
    More specifically, the $ \sim \mathcal{O}(\sqrt{\eta^{K} + \frac{N}{M}+ \frac{1}{N}})$ convergence is harder to examine since our conclusion is based on a more restrictive assumption, and we only know that there exists $\rho$ such that $\eta <1$ for the contraction to take place. In this regard, we simply pick $\eta \sim \mathcal{O}(1-\rho)$ to be a large constant, in this case $0.995$ and ensure that the convergence will take place when the batch size/iteration number is large. To study the convergence behavior, we take the following relationship between $K,M,N$: let $N=\lfloor 1.0/(\eta^K) \rfloor$ where $K$ is the number of iteration, so that we are trying to roughly look at the order  $ \sim \mathcal{O}(\sqrt{\frac{1}{N} + \frac{N}{M} + \frac{1}{N}})$. As such, we note that after choosing $K$, by taking $M\sim \mathcal{O}(N^2)$, one should expect the convergence to take place and the convergence should be of half order. This is especially the case when the diffusion is controlled. Lastly, we remark that the while for the projection approach, it is viable to only use plain SGD (one sample) as optimization procedure, for the designed damped-contraction approach, a reasonable batch size is needed (see examples below). 
\end{enumerate}
We note that a higher order scheme is used for SDE simulation for the projection algorithm and we used Method (4) from \cite{Weinan1}, section 7.5 which is of order 2. The explicit form is given in the appendix. 

\subsection{Example 1, controlled diffusion.}
 	Consider the following loss function where $d=2$        
\begin{equation}
 		J[u]=\frac{1}{2} \int_0^T \sum^d_{i=1}\mathbb{E}[(x^i-x^{i,*})^2 ]dt +\frac{1}{2} \int_0^T \sum^d_{i=1} u^{i,2}(t) dt + \frac{1}{2} \sum^{d}_{i=1} (x_T^i)^2
\end{equation}
where the forward process is 
 	\begin{equation}
 		d x^i(t) = u^i(t) -r^i(t) dt + \sigma u^i(t) dW_t \  \  	
    \end{equation}
    so that the diffusion term is controlled. 
And one needs to find $u \in \mathcal{U}$ such that 
$$J(u^*)=\min_{u\in \mathcal{U}} J(u).$$

To design parameters so that an analytic formula of the solution exists, we pick the function $r^i, x^{i,*}$ to be the following, and $u^i_t$ can be found accordingly: 
\begin{align}
	& r^1_t := \frac{-t^2/2}{\beta_t}, \ x^{1,*}=t+(\frac{T^2}{2 \sigma^2} - \frac{X^1_T}{\sigma^2}  ) \alpha_t , \  u^1_t= \frac{-t^2/2+T^2/2-X_T^1}{\beta_t}, \nonumber \\ 
	&r^2_t :=\frac{-\sin(t)}{\beta_t}, \ x^{2,*}=\cos(t) +(\frac{\sin(T)}{\sigma^2}- \frac{X^2_T}{\sigma^2} )\alpha_t , \ u^2_t :=\frac{-\sin(t)+\sin(T)-X^2_T}{\beta_t}. 
\end{align}
where $$\alpha_t=\ln\frac{(1+\sigma^2)+\sigma^2 T}{(\sigma^2+1)+\sigma^2(T-t)} \ , \ \beta_t=(1+\sigma^2)+\sigma^2(T-t) $$
and with $D=\ln(1+\frac{\sigma^2 T}{1+\sigma^2} )/ \big(\sigma^2+\ln(1+\frac{\sigma^2 T}{1+\sigma^2} )  \big) $
The parameters are set to be $x_0=1.0, \sigma=0.5$, $T=1.0$. 
\subsubsection{The projection scheme with batch samples.}
We solve the problem using projection algorithm with batch samples. In this example (Figure \ref{fig:1dcompare}), to obtain the numerical solutions for the top two figures, we take the total of temporal discretization $N=60$ and use batch size $M=N$. The total number of iterations is taken to be $N^2=3600$. The red curves are the exact solutions. To see that the predicted decay rate of $\sim \mathcal{O}(\sqrt{\frac{1}{K}+ \frac{1}{N^2}})$, we note that: 
\begin{enumerate}
    \item  In the bottom left figure, with $M=N$, we set $K$ to be a constant multiple of $N$, and see that the error demonstrates a half order decay. And it is inline with expectation. 
    \item In the bottom right figure, with $M=N$, we set $K$ to be a constant multiple of $N^2$, and see that the error demonstrates a first order decay. This matches the $\mathcal{O}(\sqrt{\frac{1}{K}+ \frac{1}{N^2}})$ analytical result. 
\end{enumerate}

\begin{figure}[h]
    \centering 
    \includegraphics[width=1.0\textwidth]{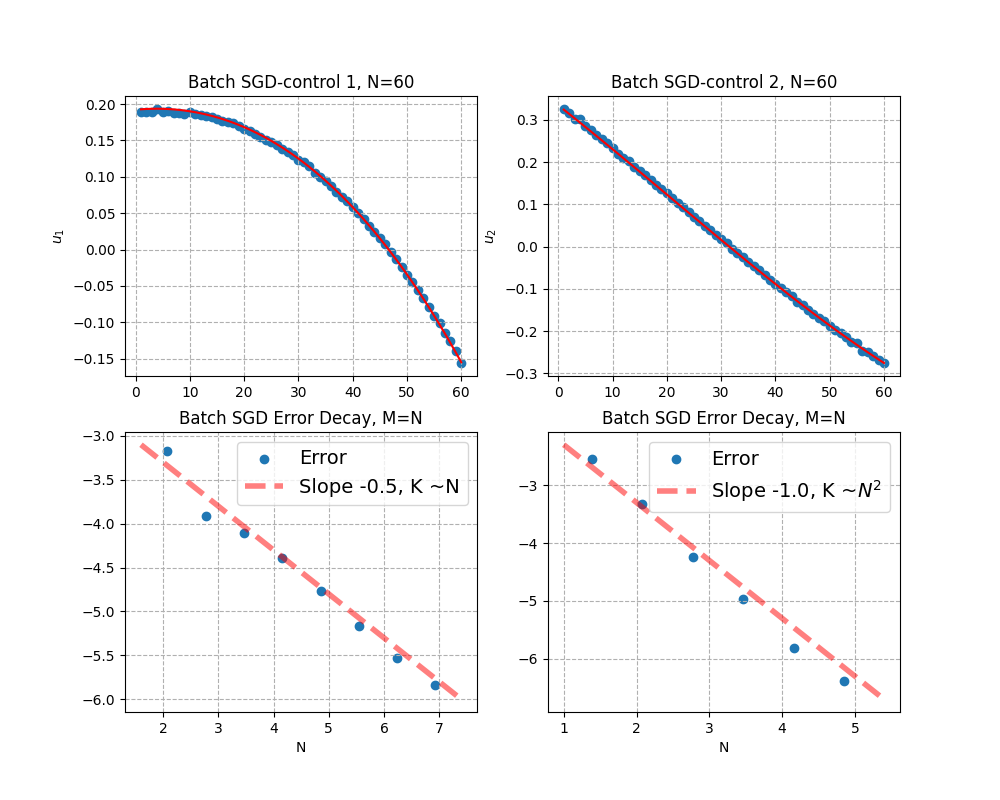} 
    \caption{Top two figures: Demonstration of numerical solutions using the Batch projection algorithm. Bottom two figures: demonstration of error decay by using different number of iterations for $K$.}
    \label{fig:1dcompare} 
\end{figure}

\subsubsection{The damped contraction scheme.}
We now solve the problem using damped contraction algorithm with batch samples. In this example (Figure \ref{fig:2dcompare}), to obtain the numerical solutions for the top two figures, we take the total of temporal discretization $N=55$ and use batch size $M=N^2$. The total number of contraction steps is taken to be $800$. 

It is noted that under some strong assumption (Assumption \ref{Contraction_setup_assumption}), the convergence rate is $\sim \mathcal{O}(\sqrt{\eta^{L} + \frac{N}{M} + \frac{1}{N}})$ where $\eta = C(1-\rho^2)+\rho^2$. Since we only know that there exists $\rho$ such that $\eta <1$ so the contraction will take place, it is very difficult to test the exact convergence rate. In this case, we simply pick $\eta \sim \mathcal{O}(1-\rho)$ to be a large constant, in this case $0.995$ and ensure that the convergence will take place when the batch size is large. More specifically, we let $N=\lfloor 1.0/(\eta^K) \rfloor$ where $K$ is the number of iteration, so that we are expecting the order  $ \sim \mathcal{O}(\sqrt{\frac{1}{N} + \frac{N}{M} + \frac{1}{N}})$.  With this decay rate under the idealized situation in mind, we note that: 
\begin{enumerate}
    \item  In the bottom right figure, with $M=N$, as we increase the number $K$ (we let $K$ takes values in $\lbrace 500 ,600, 700, 800,900 \rbrace$) hence $N$, the error first decays slowly, but then it hits a plateau: there was almost no decrease in error when $N$ was increased from 55 to 91. This shows that one will typically need at least $M>N$ for the convergence to take place which agrees with our analysis. 

    \item In the bottom left figure, with $M=N^2$, we observe that the plateau is gone and the error even demonstrates a first order decay which is faster than our anticipated half-order decay. 
\end{enumerate}

\begin{figure}[h]
    \centering 
    \includegraphics[width=1.0\textwidth]{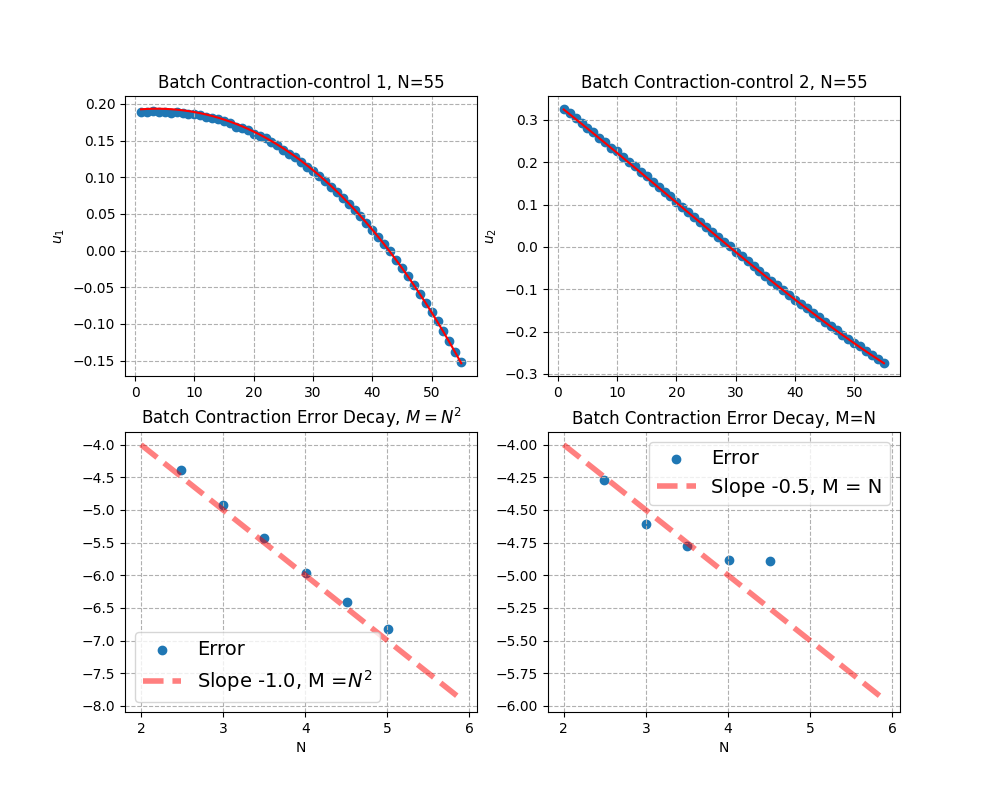} 
    \caption{Top two figures: Demonstration of numerical solutions using the damped contraction algorithm. Bottom two figures: demonstration of error decay by using different number of batch size $M$. $x-$axis is on log scale.}
    \label{fig:2dcompare} 
\end{figure}

\subsubsection{Comparison between different algorithms}
In this subsection, we compare the time efficiency of the three algorithms, namely the original stochastic gradient descent, the newly proposed batch stochastic gradient descent and the damped contraction mapping approach. The result is summarized in Table \ref{eg1_table}.

For batch SGD, we pick the batch size according to our analysis results $M=N$, and to obtain the first order convergence, we pick $K \sim \mathcal{O}(N^2)$. For Damped contraction, we pick $M \sim \mathcal{O}(N^2)$ to ensure the error decay. We then experiment with iteration size $K$ so that the relative error will be on par with that of the first two algorithms. All relative errors are computed based on averages of 50 independent runs. 

The comparison results show that for the current example, both the Batch SGD and the damped contraction methods leverages less number of iterations to achieve similar error level as the original SGD approach. Also, the improvement in efficiency can be observed from the less total time used. The damped contraction mapping approach is observed to utilize slightly more time than the batch SGD approach for each run.  

\begin{table}
\caption{Example 1 Controlled diffusion, efficiency comparison.}	
\label{eg1_table}
\begin{center}
\begin{tabular}{ |c|c|c|c|c|c|} 
\hline
 Method & $M$ (Batch size) & $K$ (Iterations) &$N$ & Time (s) & Relative Error  \\
\hline
SGD \cite{Hui1} &1 & $40^3$ &40  & 34.1s & 0.0037    \\
\hline
Batch SGD & $N$&$40^2$ & 40 &1.03 s & 0.0033 \\
\hline
Damped Contraction& $N^2$ & $710$ & 40 & 1.96s & 0.0035 \\
\hline
\end{tabular}
\end{center}
\end{table}
\subsection{Example 2, uncontrolled diffusion.}
 	Consider the following loss function where $d=2$        
\begin{equation}
 		J[u]=\frac{1}{2} \int_0^T \sum^d_{i=1}\mathbb{E}[(x^i-x^{i,*})^2 ]dt +\frac{1}{2} \int_0^T \sum^d_{i=1} u^{i,2}(t) dt 
\end{equation}
where the forward process is 
 	\begin{equation}
 		d x^i(t) = x^i(t)u^i(t) dt + \sigma x^i(t) dW_t \  \  	
    \end{equation}
    and that the diffusion term is uncontrolled. 
The goal is to find $u \in \mathcal{U}$ such that 
$$J(u^*)=\min_{u\in \mathcal{U}} J(u).$$
where we have $x^{*,i}, u^{*,i}$ to be the following: 
\begin{align}
    x^{1,*} = \frac{e^{\sigma^2 t} - (T-t)^2}{\frac{1}{x_0} -T t +\frac{t^2}{2}} +1 , \ u^{1,*}(t) = \frac{T-t}{\frac{1}{x_0} -Tt +\frac{t^2}{2}} \nonumber\\
    x^{2,*} = \frac{e^{\sigma^2 t} - (e^{-T} -e^{-t})^2}{\frac{1}{x_0} +1 -e^{-t}-te^{-T}} -e^{-t} , \ u^{2,*}(t) = \frac{e^{-T}-e^{-t}}{\frac{1}{x_0} +1 -e^{-t}-te^{-T}}
\end{align}
The parameters are set to be $x_0=1.0, \sigma=0.5, T=1.0$. 

\subsubsection{The projection scheme with batch samples.}
We solve the problem using projection algorithm with batch samples with high order scheme for the forward SDE. In this example (Figure \ref{fig:3dcompare}), to obtain the numerical solutions for the top two figures, we take the total temporal discretization $N=60$ and use batch size $M=N$. The total number of iterations is taken to be $N^2=3600$. To see that the predicted decay rate of $\sim \mathcal{O}(\sqrt{\frac{1}{K}+ \frac{1}{N^2}})$, we note that: 
\begin{enumerate}
    \item  In the bottom left figure, with $M=N$, we set $K$ to be a constant multiple of $N$, and see that the error demonstrates a first order decay, this is faster than the expected half order decay.  The main reason why a half order decay is not observed is that this is an example where the diffusion is not controlled. The $N$ term which is related to the variance of $Z_n$ is no longer present in the gradient ($j'(u) = b'_uY + f'_u $) and so it has much smaller variance. We also comment that the $\sim \mathcal{O}(\sqrt{\frac{1}{K}+ \frac{1}{N^2}})$ error decay is in fact more specifically $\sim \mathcal{O}(\sqrt{\frac{N}{M K}+ \frac{1}{N^2}})$ from the proof of Theorem \ref{standard_new}. And when we have uncontrolled diffusion, meaning $Z$ term is not present, the `$N$' in $\frac{N}{M K}$ disappears, leading to $\sim \mathcal{O}(\frac{1}{N^2})$ when $K=N$.
    
    \item In the bottom right figure, with $M=N$, we set $K$ to be a constant multiple of $N^2$, and see that even though under the same number of discretization, the relative error is smaller than the case when $K\sim \mathcal{O}(N)$ (bottom left figure), the error still demonstrates a first order decay. This is inline with the analysis for the discretization error.  
\end{enumerate}

\begin{figure}[h]
    \centering 
    \includegraphics[width=1.0\textwidth]{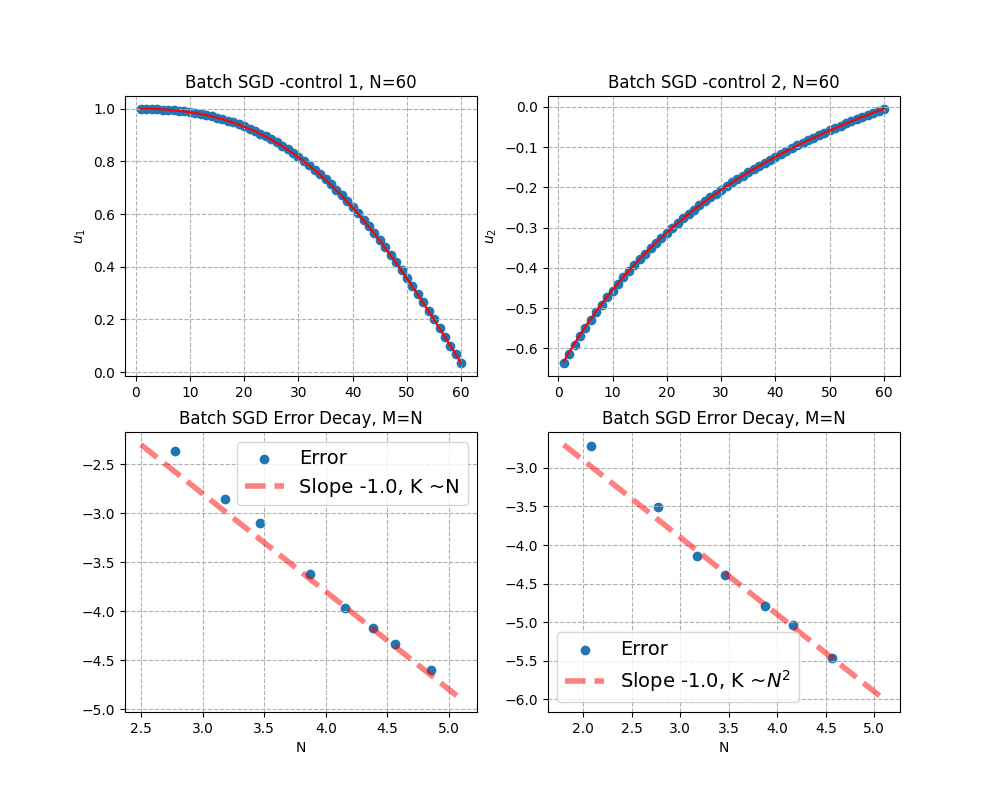} 
    \caption{Top two figures: Demonstration of numerical solutions using the Batch projection algorithm. Bottom two figures: demonstration of error decay by using different number of iterations for $K$. $x-$axis in log scale.}
    \label{fig:3dcompare} 
\end{figure}

\subsubsection{The damped contraction scheme.}
We now solve the problem using damped contraction algorithm with batch samples. In this example (Figure \ref{fig:4dcompare}), to obtain the numerical solutions for the top two figures, we take the total of temporal discretization to be $N=55$ and use batch size $M=N^2$. The total number of contraction steps is taken to be $800$. 

As in the case in Section 5.1.2, we pick $\eta \sim \mathcal{O}(1-\rho) $ to be a large constant, in this case $0.995$ and ensure that the convergence will take place when the batch size is large. More specifically, we let $N=\lfloor 1.0/(\eta^K) \rfloor$ where $K$ is the number of iteration, so that we are trying to roughly look at the order  $ \sim \mathcal{O}(\sqrt{\frac{1}{N} + \frac{N}{M} + \frac{1}{N}})$. We note that: 
\begin{enumerate}
    \item Since in this case the diffusion is not controlled, the $N$ term related to $Z_n$ is not present hence one should expect a rate of $ \sim \mathcal{O}(\sqrt{\frac{1}{N} + \frac{1}{M} + \frac{1}{N}})$ instead. Thus, when one pick $M=N$, half order decay is expected in terms of $N$. In the bottom left figure it is noted that since $N$ term is gone, convergence takes place when $M \sim \mathcal{O}(N)$. 
    \item In the bottom right figure, with $M=N^2$, we observe that the algorithm demonstrates a first order decay which is an improvement over case than when $M=N$. This is  faster than the half order decay we anticipated. 
\end{enumerate}

\begin{figure}[h]
    \centering 
    \includegraphics[width=1.0\textwidth]{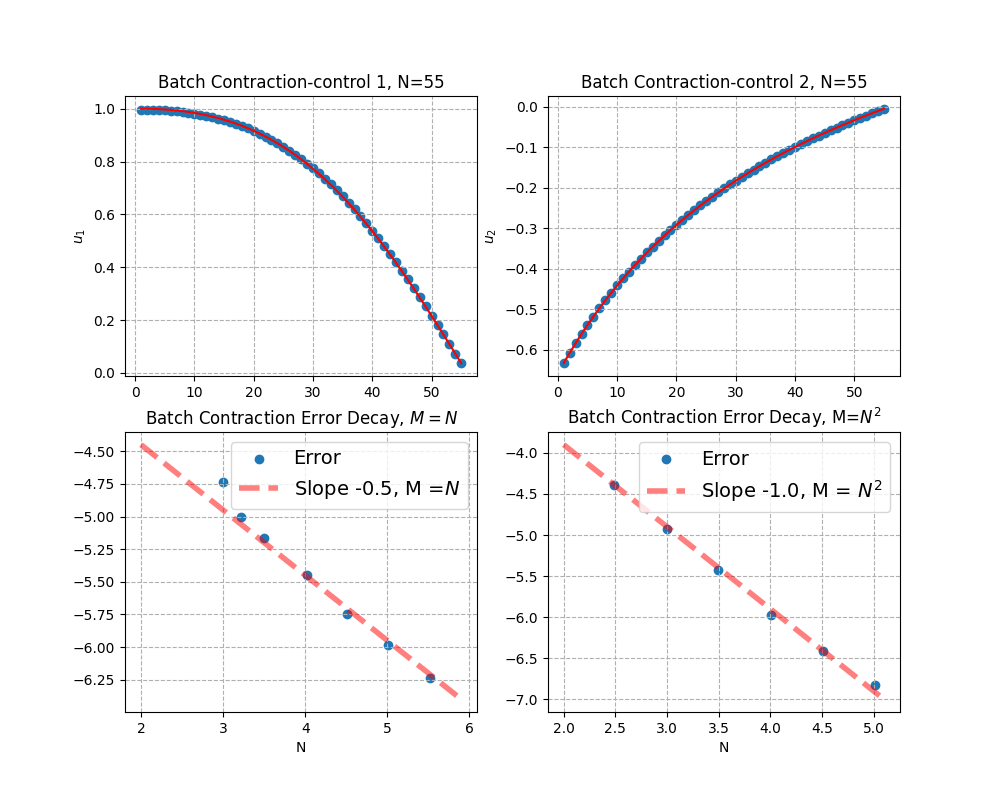} 
    \caption{Top two figures: Demonstration of numerical solutions using the damped contraction algorithm. Bottom two figures: demonstration of error decay by using different batch sizes.}
    \label{fig:4dcompare} 
\end{figure}

\subsubsection{Comparison between different algorithms}
In this section, we compare the time efficiency of the three algorithms for the uncontrolled diffusion example.  The result is summarized in Table \ref{eg2_table}.

Similar to the choice in section 5.1.3, for batch SGD, we pick the batch size according to our analysis results $M=N$, and to obtain the first order convergence, we pick $K \sim \mathcal{O}(N^2)$. For Damped contraction, we pick $M \sim \mathcal{O}(N^2)$ to ensure the error decay. We then experiment with iteration size $K$ so that the relative error will be on par with that of the first two algorithms. All relative errors are computed based on averages of 50 independent runs. The conclusion reached is similar to that from Section 5.1.3.

\begin{table}
\caption{Example 1 Controlled diffusion, efficiency comparison}	
\label{eg2_table}
\begin{center}
\begin{tabular}{ |c|c|c|c|c|c|} 
\hline
 Method & $M$ (Batch size) & $K$ (Iterations) &$N$ & Time (s) & Relative Error  \\
\hline
SGD &1 & $40^3$ &40  & 29.14s & 0.00429    \\
\hline
Batch SGD & $N$&$40^2$ & 40 &0.91 s & 0.00426 \\
\hline
Damped Contraction& $N^2$ & $620$ & 40 & 1.81s & 0.00420\\
\hline
\end{tabular}
\end{center}
\end{table}

\subsection{Solving high dimensional HJB using Batch projection algorithm}
In this section, we will use the Batch projection algorithm to solve a high dimension HJB equation used in the seminal paper \cite{Weinan2}. The equation of interest is as follows:
\begin{align}
    \begin{cases}
    \frac{\partial v}{\partial t} + \Delta_x v - \lambda |D_x v|^2 = 0 ,& (t,x) \in [0,T) \times \mathbb{R}^d \\
    v(T,x)=g(x),              & x \in \mathbb{R}^d
\end{cases}
\end{align}
which originates from the following HJB equation
\begin{align}
    \begin{cases}
    \frac{\partial v}{\partial t} + \Delta_x v + \lambda \inf_{a \in \mathbb{R}^d} [|a|^2 + 2a \cdot D_x v ] = 0 ,& (t,x) \in [0,T) \times \mathbb{R}^d \\
    v(T,x)=g(x),              & x \in \mathbb{R}^d.
\end{cases}
\end{align}
This equation is related to the following stochastic optimal control problem: 
\begin{align}
	v(t,x) = \inf_{\alpha \in \mathcal{A}} \mathbb{E}[\int^T_t |\alpha_s|^2 ds + g(X^{t,x,\alpha}_T)]
\end{align}
and $X_s=X_s^{t,x,\alpha}$ is the controlled process governed by 
\begin{align}
	dX_s=2 \sqrt{\lambda} \alpha_s ds + \sqrt{2}dW_s,  \ \ \ t \leq s \leq T, \ X_t=x. \label{diff_hjb}
\end{align}
Note that the optimal control of the problem is of a feedback form, i.e $\alpha_s=u_s(x)$ for some $u_s : \bR^d \rightarrow \bR^d$. Thus, to find the optimal control, after discretizing the time interval $[0,T]$ with uniform mesh of size $\frac{T}{N}$, we approximate $u_n(x)$ with a a randomized neural network: 
\begin{align}
    u_n(x)=A_n \tilde{\phi}(\tilde{A}_n x+\tilde{b}_n) +b_n 
\end{align}
where $x\in \bR^d, A_n \in \bR^{d_1 \times \tilde{d}}, \tilde{A}_n \in \bR^{\tilde{d} \times d}, b_n \in \bR^{d_1}$ with  $d_1 =d=100$. 
The SDE \eqref{diff_hjb} then finds the following discretization: 
\begin{align}\label{diff_hjb_disc}
    X_{n+1}=X_n+ 2 \sqrt{\lambda} u_n (X_n) \Delta t + \sqrt{2} \Delta W_n
\end{align}

We emphasize a few considerations of using the Randomized Neural network structure: 
\begin{enumerate}[i.)]
    \item We use the randomized neural network with only one hidden layer and the only trainable parameters are $(A_n, b_n)$ for $n=0,...,N-1$. 
    \item The fixed parameters $(\tilde{A}_n, \tilde{b}_n)$ are sampled independently from Gaussian distribution with zero mean and preselected variance. 
    \item Given this special structure, the gradient of the Hamiltonian with respect to both the input variable $x$ and the weight vectors $(A_n, b_n)$ can be computed manually. This avoids the full error backpropagation of the loss under the approach of \cite{Jiequn1}. Under our SMP framework, we see that the coefficient is actually linear in the control parameters (trainable weights).
\end{enumerate}

Now, the Hamiltonian is defined accordingly as 
\begin{align}
    H_n= 2 \sqrt{\lambda} Y^T_n u_n(x)+ \sqrt{2} tr(I Z_n ) + |u_n|^2
\end{align}
and $(Y_n, Z_n)$ are defined via the following system of equations: 
\begin{align}
\begin{cases}
        Y_{n} & = Y_{n+1}+ 2 \sqrt{\lambda} Y^T_{n+1} \Big( A_n \cdot \text{diag}(\tilde{\phi}_n')\cdot \tilde{A}_n + 2 u_n(x)^T \big( A_n \cdot \text{diag}\big(\tilde{\phi}_n'\big)\cdot \tilde{A}_n \big) \Big) \Delta t \nonumber \\ 
        Z_{n} & = \frac{Y_{n+1} \Delta W^T_n }{\Delta t}
\end{cases}
\end{align}
where we write $\tilde{\phi}_n'$ to represent $\tilde{\phi}':= \tilde{\phi}'(\tilde{A}_n x+\tilde{b}_n)$. 
The sample-wise gradient $j'(A_n,b_n)$ is then computed as follows: 
\begin{align}\label{grad_rand_nn}
    j'(A_n) &= 2 \sqrt{\lambda} Y_n \tilde{\phi}_n^T + 2(A_n \cdot \tilde{\phi}_n +b) \tilde{\phi}^T \nonumber \\
    j'(b_n) &= 2 \sqrt{\lambda} Y_n^T + (A_n \cdot \tilde{\phi}_n + b_n)^T  
\end{align}
where we have used the shorthand notation $\tilde{\phi}_n$ to denote $\tilde{\phi}(\tilde{A}_n x+\tilde{b}_n)$. With the explicit form of the gradient above, one can perform batch gradient decent accordingly (Algorithm \ref{algorithm_batch_sample}). 

\subsubsection{Numerical results for HJB}
In this section, we solve and compare the solution of the HJB equation at a single space-time location $(t_0, x_0)=(0, 0)$ obtained via the numerical methods discussed in this section. The exact solution has the following analytic expression and we find $v(0,0)$ via Monte Carlo simulation.  
\begin{align}
	v(t,x) =-\frac{1}{\lambda} \ln \Big(\mathbb{E}[\exp\big(-\lambda g(x+\sqrt{2}W_{T-t})\big)]\Big),  (t,x) \in [0,T] \times \mathbb{R}^d \label{exact_hj}
\end{align}
Since we are now incorporating SMP framework with the training of neural networks, we can also use alternative optimizers for updating the trainable parameters (controls) such as SGD/Adam/Adagrad. We summarize and compare the result among these three approaches in the figure below. For each optimizer, we use batch size equal to 1024 and train for total 380 epochs. The activation functions are all picked to be the hyperbolic tangent function ($tanh$) function. The test is implemented with 128 hidden neural nets with $\lambda=1, d=100$. 

The top figure in Figure \ref{fig:HJBcompare} shows one typical training instance for each update procedure. It is noted that the Adam and AdaGrad optimizers demonstrate faster convergence than SGD. After running each method independently for 20 test cases, we find the following percentage error: SGD-0.35\%, AdaGrad-0.06\%, Adam-0.035\% which shows that the Adam optimizer produces a more accurate result. 

Hence, to further test the robustness of the SMP training framework, we adopt the Adam optimizer while changing the $\lambda$'s in the equation. It is noted from the bottom figure from Figure \ref{fig:HJBcompare} that the error rate is overall acceptable: they never exceed $0.5\%$. We comment that the results are not finetuned, namely we only used constant learning rate $2\times 10^{-3}$ one hidden layer of dimension 128 for all cases. To achieve better accuracy, one is encouraged to change those variables and use learning rate scheduling.  
\begin{figure}[h]
    \centering 
    \includegraphics[width=0.8\textwidth]{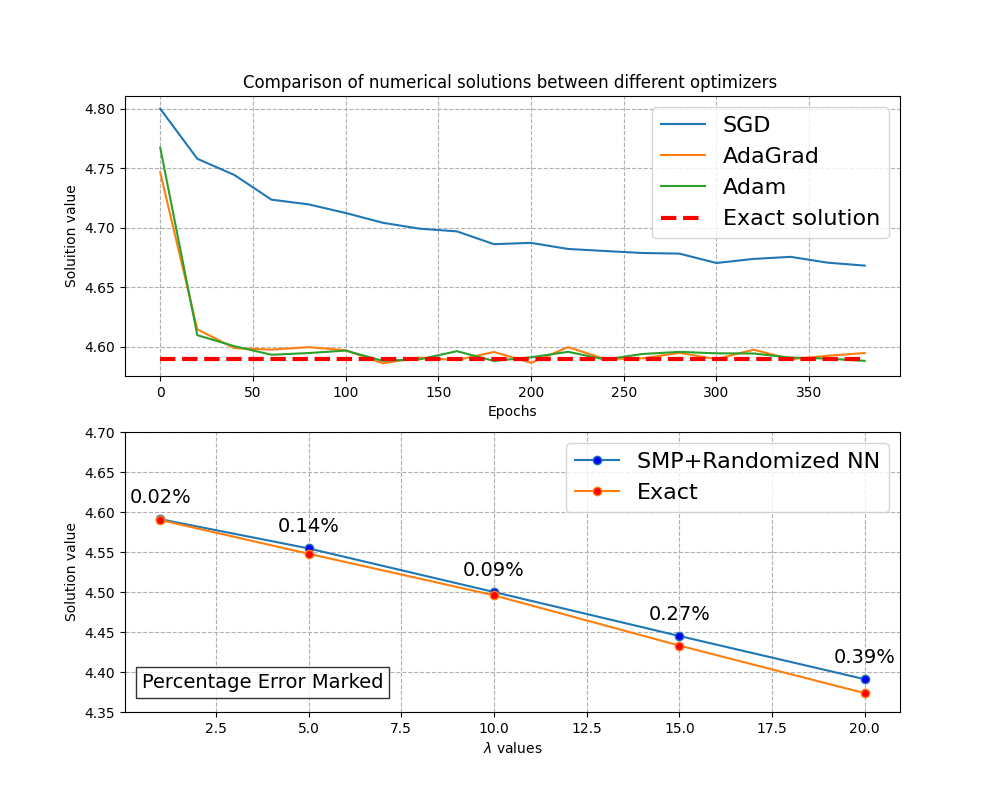} 
    \caption{Top figure: Comparison between numerical solutions obtained via different optimizers. Bottom figure: Comparison between numerical and exact solutions by using the Adam optimizer under different $\lambda$'s.}
    \label{fig:HJBcompare} 
\end{figure}

\section{Conclusion and future outlook}
In this work, we designed and updated the algorithm proposed in \cite{Hui1} and improved the convergence rate in \cite{Hui1} from $\mathcal{O}(\sqrt{\frac{N}{K}+\frac{1}{N}})$  to $\mathcal{O}(\sqrt{\frac{1}{K}+\frac{1}{N^2}})$ under the strong convexity assumption. Under some specific setup, we further designed a damped contraction algorithm that leverages both the sufficient and necessary condition of the Stochastic Maximum principle. Our experiments show the two algorithms under the batch sample setup demonstrate improvement in time efficiency over the original SGD approach. In particular, the batch SGD is shown to be more efficient. Our last example then demonstrates that the batch SGD algorithm can be used to solve practical problems of much higher dimensions. Also, when randomized neural networks are used to approximate the feedback control, the whole training process is more interpretable in the sense that the complicated error back propagation can be avoided, and the gradients can be computed algebraically. 

We further note that to better improve the efficiency/accuracy of the algorithm, the pathwise realization for $Z_t$ in the solution $(Y_t,Z_t)$ of the BSDE needs to be further designed for variance reduction. Also, we will further explore the possible extension of the designed contraction mapping based algorithm to the training of deep neural networks so that the entire neural network traning will be derivative free.

\section{Appendix}

Proof for \textbf{Proposition \ref{stab_x}}. 
\begin{proof}
	By definition, we have  
	\begin{align}
		\xuv_{n+1}&=\xuv_n + \Delta b^{u,v}_n h + \Delta \sigma^{u,v}_n \Delta W_n
	\end{align}
	Squaring both sides and take expectations, we obtain the following
\begin{align}
	\E[|\xuv_{n+1}|^2] &= \E[|\xuv_n + \Delta b^{u,v}_n h|^2]+h\E[|\Delta \sigma^{u,v}_n|^2]\\
	& \leq \E[(1+\epsilon)|\xuv_n|^2+ (1+\epsilon^{-1})Ch^2(|\xuv_n|^2+|u_n-v_n|^2)] \\
	&+ Ch|u_n-v_n|^2 \nonumber \\
	&\leq(1+\epsilon+(1+\epsilon^{-1})Ch^2) \E[|\xuv_n|^2]+\Big((1+\epsilon^{-1})Ch^2+Ch \Big)|u_n-v_n|^2)\\
	&\leq (1+Ch)\E[|\xuv_n|^2]+Ch|u_n-v_n|^2
\end{align}
where we have taken $\epsilon=h$. 
Then, by discrete Gronwall's inequality, we have 
\begin{equation}\label{xuv_estimate}
	\max_{0 \leq n \leq N} |\xuv_n|^2 \leq Ch \sum^{N-1}_{n=0}|u_n-v_n|^2
\end{equation}
\end{proof}

Proof for \textbf{Proposition \ref{stab_yz}}
\begin{proof}
Recall the definition of $Y^{N,u}_{t_{n+1}}$. By the Martingale representation theorem, there exists a square integrable process $\bar{\zu_t}$ such that 
\begin{equation}
	Y^{N,u}_{t_{n+1}}=\E_{t_n}[Y^{N,u}_{t_{n+1}}]+\int^{t_{n+1}}_{t_n} \bar{\zu_t}d W_t
\end{equation}
Hence, we have:
\begin{align}
	Z^{N,n}_{t_n}&=\frac{1}{h}\E_{t_n}[Y^{N,u}_{n+1}\Delta W_{n}] \label{z1}\\
	&=\frac{1}{h}\E_{t_n}[\int^{t_{n+1}}_{t_n} \bar{\zu}_td W_t \Delta W_{n+1}]\label{z2}\\
	&=\frac{1}{h}\E_{t_n}[\int^{t_{n+1}}_{t_n} \bar{\zu}_t dt]\label{z3}
\end{align}
Thus we have the following system of equations:
	\begin{equation}
		\begin{cases}
	Y^{N,u}_{t_{n}}=Y^{N,u}_{t_{n+1}}+h\E_{t_n}[b'_x(X^u_{n},u_{n})^T Y^{N,u}_{t_{n+1}}+f'_x(X^u_{n},u_{n})]-\int^{t_{n+1}}_{t_n} \bar{\zu_t}d W_t\\
	Y^{N,v}_{t_{n}}=Y^{N,v}_{t_{n+1}}+h\E_{t_n}[b'_x(X^v_{n},v_{n})^T Y^{N,v}_{t_{n+1}}+f'_x(X^v_{n},v_{n})]-\int^{t_{n+1}}_{t_n} \bar{\zv_t}d W_t
		\end{cases}
	\end{equation}
Take the difference between the two equations above we get 
\begin{equation}{\label{dfbsde_diff}}
	\yuv_{t_n}=\yuv_{t_{n+1}}+h\E_{t_n}[b_x(X^u_{n},u_{n})^T Y^{N,u}_{t_{n+1}}-b_x(X^v_{n},v_{n})^T Y^{N,v}_{t_{n+1}} +f'_x(X^u_{n},u_{n})-f'_x(X^v_{n},v_{n})]-\int^{t_{n+1}}_{t_n} (\bar{\zu_t}-\bar{\zv_t})d W_t
\end{equation}
Let $\iuv_{n+1}:=\E_n[b'_x(X^u_{n},u_{n})^T Y^{N,u}_{t_{n+1}}-b'_x(X^v_{n},v_{n})^T Y^{N,v}_{t_{n+1}}+f'_x(X^u_{n},u_{n})-f'_x(X^v_{n},v_{n})]$, then recall \eqref{z1}-\eqref{z3}, we have 
\begin{align}
\zu_{t_n}-\zv_{t_n}&=\frac{1}{h}\E_{t_n}[\int^{t_{n+1}}_{t_n} \bar{\zu_t}-\bar{\zv_t}d t] \\ 
\Rightarrow  |\zuv_n|^2&\leq \frac{1}{h} \E_{t_n}[\int^{t_{n+1}}_{t_n} |\bar{\zuv_t}|^2 dt ] \label{dz}
\end{align}
Now, by writing \eqref{dfbsde_diff} as 
\begin{equation}
	\yuv_n+\int^{t_{n+1}}_{t_n} (\bar{\zu_t}-\bar{\zv_t})d W_t=\yuv_{n+1}+h\E_n[f'_x(X^u_{n},u_{n})^T Y^{N,u}_{t_{n+1}}-f'_x(X^v_{n},v_{n})^T Y^{N,v}_{t_{n+1}}+f'_x(X^u_{n},u_{n})-f'_x(X^v_{n},v_{n})]
\end{equation}
squaring both sides, take expectation and use \eqref{dz}, we obtain the following inequality: 
\begin{align}
	\E[|\yuv_n|^2 + h|\zuv_n|^2] &\leq \E[|\yuv_{n+1}+ h \iuv_{n+1}|^2]\\
	&\leq (1+\frac{h}{\epsilon} )\E[|\yuv_{n+1}|^2]+(1+\frac{\epsilon}{h})h^2\E[|\iuv_{n+1}|^2]  \label{ed1}
\end{align}
Now we consider estimates for $|\iuv_{n+1}|^2$:
\begin{align}
	\iuv_{n+1}&=\E_{t_n}[b'_x(X^u_{n},u_{n})^T Y^{N,u}_{t_{n+1}}-b'_x(X^v_{n},v_{n})^T Y^{N,v}_{t_{n+1}}+f'_x(X^u_{n},u_{n})-f'_x(X^v_{n},v_{n})] \nonumber \\
	&=\E_{t_n}[ \Big(b'_x(X^u_{n},u_{n})-b'_x(X^v_{n},v_{n})\Big)^T Y^{N,u}_{t_{n+1}}-f_x(X^v_{n},v_{n})^T(Y^{N,v}_{t_{n+1}}-Y^{N,u}_{t_{n+1}}) \nonumber \\
	& +f'_x(X^u_{n},u_{n})-f'_x(X^v_{n},v_{n})] \nonumber \\
	& \leq \E_n[C(|\xuv_{n}|+|u_{n}-v_{n}|) Y^{N,u}_{n}]+\E_n[|f_x(X^v_{n},v_{n}) \yuv_{n+1}|]+\E_n[|u_{n}-v_{n}|+|\xuv_{n}|] \nonumber
\end{align}
use Cauchy's inequality on the first term, squaring both sides and take expectation, we get 
\begin{align}
	\E[\iuv_{n+1}|^2] &\leq C \E[|\xuv_{n}|^2 + |u_{n}-v_{n}|^2] \sup_n \E[|Y^{N,u}_{t_n}|^2]+ C\E[|\yuv_{n+1}|^2]+C\Big(\E[|\xuv_{n}|^2 + |u_{n}-v_{n}|^2] \Big) \nonumber \\
	&\leq C \Big( \E[|\xuv_{n+1}|^2 + |u_{n}-v_{n}|^2] \Big) + C\E[|\yuv_{n+1}|^2] \label{ed2}
\end{align}
where we use the assumption that $b'_x$ is uniformly bounded, so that both $\E[|b'_x(X^v_{n},v_{n})|^2]$ and $\sup_n \E[|Y^{N,u}_{n}|^2]$ are uniformly bounded by a constant $C$ independent of the control $u$. 
Combining \eqref{ed1} and \eqref{ed2}, we obtain the following inequality: 
\begin{align}
	\E[|\yuv_n|^2 + h|\zuv_n|^2] &\leq (1+\frac{h}{\epsilon}+(h+\epsilon)Ch) \E[|\yuv_{n+1}|^2]+h(h+\epsilon)C \E[|\xuv_{n}|^2 + |u_{n}-v_{n}|^2] \nonumber \\
	&\leq (1+Ch) \E[|\yuv_{n+1}|^2]+Ch\E[|\xuv_{n}|^2|+|u_{n}-v_{n}|^2|] \label{tosumup}
\end{align}
where we picked $\epsilon=1/C$ and used the Proposition \ref{stab_x}.  Thus, by the discrete Gronwall's inequality (backward), and the assumption that the terminal function is also Lipschitz, we have
\begin{align}
	\sup_{0 \leq n \leq N} \E[|\yuv_n|^2 ] &\leq Ch \sum^{N-1}_{n=0} \E[|\xuv_{n}|^2|+|u_{n}-v_{n}|^2] \nonumber \\
	& \leq Ch\sum^{N-1}_{n=0} |u_{n}-v_{n}|^2 \label{dy_estimate}
\end{align}
This finishes the estimate for $\sup_{0 \leq n \leq N} \E[|\yuv_n|^2 ]$. 

To obtain an estimate for $h\sum^{N-1}_{n=0} \E[|\zuv_n|^2]$, we sum up the inequality \eqref{tosumup} on both sides and use \eqref{dy_estimate}, we obtain. 
\begin{equation}{\label{dz_estimate}}
	h\sum^{N-1}_{n=0} \E[|\zuv_n|^2] \leq Ch\sum^{N-1}_{n=0} |u_{n}-v_{n}|^2
\end{equation}
Combining \eqref{dz_estimate} and \eqref{dy_estimate}, we have
\begin{equation}
	\sup_{0\leq n\leq N}\E[|\yuv_n|^2]+h\sum^{N-1}_{n=0} \E[|\zuv_n|^2]  \leq Ch\sum^{N-1}_{n=0} |u_{n}-v_{n}|^2
\end{equation}
\end{proof}

Proof for \textbf{Proposition \ref{continuous_YZbound}}.
\begin{proof}
 We take the It\^{o} derivative of the term $|Y^u_t-Y^v_t|^2$ which are defined as follows:
\begin{align}
   dY^u_t &= -(b'_x(u_t) Y_t^u +  f'_x(x^u_t)) dt + Z_t^u dW_t \nonumber \\ 
   dY^v_t &= -(b'_x(v_t) Y_t^v+ f'_x(x^v_t)) dt + Z_t^v dW_t 
\end{align}
The following results are attained:  
\begin{align}
d|\yuv_t|^2 &= 2Y^{u,v}_t dY^{u,v}_t + |\Delta  Z^{u,v}_t |^2 dt  
\end{align}
Integrating both sides and take expectation, the following equality is attained: 
\begin{align}
    &\E[|\yuv_t|^2+ \int^T_t |\zuv_s |^2 ds]  \leq |\yuv_T|^2+2 \E[\int^T_t \yuv_s \Big ((\Delta b^{u,v}_x)_s Y^u_s + (b^v_x)_s \yuv_s  + \Delta f^{u,v}_s \Big ) ds]  \nonumber\\ 
    & \leq \E[|\yuv_T|^2+ 2 \int^T_t |(\Delta b^{u,v}_x)_s| |\yuv_s Y^u_s| + |(b^v_x)_s| |\yuv_s|^2 + \Delta f^{u,v}_s \yuv_s ds ]\nonumber \\ 
    & \leq \E[|\yuv_T|^2]+ \E[ (\sup_{s}|Y^u_s|)^2] \int^T_t  (\Delta b^{u,v}_x)_s  ds + \E[\int^T_t |\yuv_s |^2 ds] +C \E[\int^T_t |\yuv_s |^2 ds \nonumber \\ 
    &+\E[\int^T_t |\Delta f^{u,v}_s|^2 ds]+ \E[\int^T_t |\Delta Y^{u,v}_s|^2 ds] \label{ineq_sup} \\ 
    & \leq \E[|\yuv_T|^2] + C \int^T_t |u_s-v_s|^2 ds + C \E[\int^T_t |\xuv_s|^2ds] +C\E[\int^T_t |\yuv_s |^2 ds
\end{align}
Thus, by using the (backward) Gronwall's inequality, we obtain: 
\begin{align}
    \sup_{t\in [0,T]} \E[|\yuv_t|^2] &\leq C \Big( \E[\int^T_0 |\xuv_s|^2ds] + \E[|\yuv_T|^2] \Big) \nonumber\\ 
    & \leq C \Big( T \sup_{t\in [0,T]}\E[|\xuv_t|^2] + \E[|\xuv_t|^2] + \int^T_0 |u_s-v_s|^2 ds \Big) \nonumber\\ 
    & \leq C_{YZ} ||u-v||^2_2
\end{align}
This naturally  also imply that 
\begin{align} \label{continuous_yuv_bound}
    \E[|\yuv_t|^2+ \int^T_t |\zuv_s |^2 ds] \leq C_{YZ} ||u-v||^2_2
\end{align}
for some $C_{YZ} >0$. 
\end{proof}

Proof for \textbf{Proposition \ref{contraction_stab_yz}}. 
\begin{proof}
    Consider the following systems of equations. 
    \begin{align}
        \begin{cases}
            H_t^{',u} := \E \big[ \alpha(X^u_t)Y^u_t + \beta_t Z^u_t + \gamma(X^u_t) \big] \nonumber\\
            H_t^{',v} := \E \big[ \alpha(X^v_t)Y^v_t + \beta_t Z^v_t + \gamma(X^v_t) \big]
        \end{cases}
    \end{align}
Take the difference between the two equations, square both sides and integrate from $0$ to $T$ to obtain: 
\begin{align}
    &\int^T_0 |\E \big[ \alpha(X^u_t)Y^u_t- \alpha(X^v_t)Y^v_t + \beta_t (Z^u_t-Z^v_t )+ \gamma(X^u_t)-\gamma(X^v_t) \big]|^2 dt  \nonumber\\
    &\leq \int^T_0 3|\E \big[  \alpha(X^u_t)Y^u_t- \alpha(X^v_t)Y^u_t+\alpha(X^v_t)Y^u_t - \alpha(X^v_t)Y^v_t  \big]|^2 + 3C \int^T_0 \E[|Z^u_t-Z^v_t|^2]  dt + 3 \int^T_0 \E[|\gamma(X^u_t)-\gamma(X^v_t)|^2]  dt \nonumber \\ 
    & \leq 6C \int^T_0 \E[|Y^u_t-Y^v_t|^2]dt + 6 \int^T_0 \E[|\alpha(X^u_t)-\alpha(X^v_t)|^2](\sup_{0\leq t \leq T} \E[|Y^u_t|^2])dt + C||u-v||^2 \nonumber\\ 
    & \leq C||u-v||^2
\end{align}
where we used the fact that $\alpha, \beta$ are both uniformly bounded, and $\alpha$ is Lipshitz in its argument. We also used Proposition \ref{continuous_Xbound} and Proposition \ref{continuous_YZbound}. 
\end{proof}

\textbf{High order scheme}

We list here the explicit form for the higher order scheme used for the simulation of forward SDE which is Method (4) from \cite{Weinan1} Section 7.5. We note that the scheme is of order 2: 
\begin{align}
    X_{n+1}&=X_n+b \Delta t + \sigma \Delta W_n+\frac{1}{2} \sigma \sigma'\big( (\Delta W_n)^2  -\Delta t \big) + \sigma b' \Delta Q_n + \frac{1}{2} \big( bb'+\frac{1}{2} \sigma^2 b'' \big) \Delta t^2 \nonumber\\
    &+\big( b\sigma' + \frac{1}{2}\sigma^2 \sigma'' \big)(\Delta W_n \Delta t - \Delta Q_n) + \frac{1}{2}\sigma(\sigma\sigma''+(\sigma)^2)(\frac{1}{3} \Delta W_n^2 -\Delta t)\Delta W_n
\end{align}
where $\Delta Q_n = \int^{t_{n+1}}_{t_n}\int^s_{t_n} dW_{\tau} ds$ is a Gaussian random variable satisfying $\E[\Delta Z_n]=0$, $\E[(\Delta Z_n)^2]=\frac{\Delta t^3}{3}$ and $\E[\Delta Z_n \Delta W_n]=\frac{\Delta t^2}{2}$.

\end{document}